\documentclass[a4paper,12pt]{amsart}

\usepackage{amssymb,enumerate,psfrag,graphicx,amsfonts,amsrefs,amsthm,mathrsfs,amsmath,amscd,version,graphicx}
\usepackage{tikz,cleveref,caption}
\usetikzlibrary{graphs,graphs.standard}

\usepackage{xcolor}
\usepackage{tikz-cd}
\usepackage{tikz}
\usetikzlibrary{arrows}
\tikzset{negated/.style={
        decoration={markings,
            mark= at position 0.5 with {
                \node[transform shape] (tempnode) {$\backslash$};
            }
        },
        postaction={decorate}
    }
}
\tikzset{
    vertex/.style={draw,circle,thick, inner sep=1.5pt,minimum size=6pt,
    },
    edge/.style={thick},
    dedge/.style = {->,> = latex',thick}
}
\definecolor{asparagus}{rgb}{0.53, 0.66, 0.42}
\usetikzlibrary{decorations.markings}
\usetikzlibrary{decorations.text}
\Crefname{equation}{Eq.}{Eqs.}

\newtheorem{theo}{Theorem}[section]
\newtheorem{theorem}[theo]{Theorem}
\newtheorem{lemma}[theo]{Lemma}
\newtheorem{proposition}[theo]{Proposition}
\newtheorem{corollary}[theo]{Corollary}
\newtheorem{definition}[theo]{Definition}

\theoremstyle{definition}
\newtheorem{remark}[theo]{Remark}
\newtheorem{example}[theo]{Example}

\numberwithin{equation}{section}

\begin{document}

\title{On cospectrality of gain graphs}

\author[M. Cavaleri]{Matteo Cavaleri}
\address{Matteo Cavaleri, Universit\`{a} degli Studi Niccol\`{o} Cusano - Via Don Carlo Gnocchi, 3 00166 Roma, Italia}
\email{matteo.cavaleri@unicusano.it   \textrm{(Corresponding Author)}}

\author[A. Donno]{Alfredo Donno}
\address{Alfredo Donno, Universit\`{a} degli Studi Niccol\`{o} Cusano - Via Don Carlo Gnocchi, 3 00166 Roma, Italia}
\email{alfredo.donno@unicusano.it}

\begin{abstract}
We define $G$-cospectrality of two $G$-gain graphs $(\Gamma,\psi)$ and $(\Gamma',\psi')$, proving that it is a switching isomorphism invariant. When $G$ is a finite group, we prove that $G$-cospectrality is equivalent to cospectrality with respect to all unitary representations of $G$.
Moreover, we show that two connected gain graphs are switching equivalent if and only if the gains of their
closed walks centered at an arbitrary vertex $v$ can be simultaneously conjugated. In particular, the number of switching equivalence classes on an underlying   graph $\Gamma$ with $n$ vertices and $m$ edges,  is equal to the number of simultaneous conjugacy classes of the group $G^{m-n+1}$.

We provide examples of $G$-cospectral non-switching isomorphic graphs and we prove that any gain graph on a cycle is determined by its $G$-spectrum. Moreover, we show that when $G$ is a finite cyclic group, the cospectrality with respect to a faithful irreducible representation implies the cospectrality with respect to any other faithful irreducible representation, and that the same assertion is false in general.
\end{abstract}

\maketitle

\begin{center}
\end{center}

\section{Introduction}
The spectrum of the adjacency matrix of a graph determines  many properties of the graph:
number of edges, connectivity, bipartiteness, automorphisms, etc. (see \cite{cve}).
This relation between graph structural properties  and linear algebra is the  focus of the \emph{spectral graph theory}. Indeed, one can be interested in what properties of a graph are determined by its spectrum, but also, more drastically, in which graphs are determined by their spectrum \cite{deter}. Actually the latter
 issue is old (it originated from chemistry \cite{gp}) and it is clearly related with the problem  of constructing pairs of \emph{cospectral} non-isomorphic graphs.
 Since the first examples of such a pairs \cite{Colla} to today, many works have been devoted to the subject: we point out the result of Schwenk, asserting that \emph{almost all} tree have cospectral, non-isomorphic mates \cite{tree}, and the development by Godsil and McKay of a tool to generate cospectral graphs \cite{GoMc}, known as \emph{Godsil-McKay switching}.
 \\\indent Part of the success of graphs is due to the fact that they are a model for systems of \emph{things} that have or do not have, two by two, interaction. The introduction of signed graphs  originates from the need to distinguish also between positive and negative interaction. A signed graph is a graph whose edges can be positive or negative \cite{Harary}: more precisely, it is a pair $(\Gamma,\sigma)$ where $\Gamma$ is  the \emph{underlying graph} and $\sigma$ is the \emph{signature}, that is a map from the set of the edges of $\Gamma$ to $\{\pm 1\}$.
 As  graphs are investigated up to isomorphism, signed graphs are investigated up to \emph{switching isomorphism}  \cite{zasign,zasgraph}. The notion of switching isomorphism is inspired from the Seidel's switching \cite{sei},
but it is  based on the switch of the signs of the edges incident to a fixed vertex. A \emph{balanced signed graph} is a graph  where the product of the signs of the edges along any cycle is positive or equivalently, it is a signed graph whose edges can be switched to be all positive \cite{Harary}.
A signed graph has a natural \emph{signed adjacency matrix}, a matrix whose non-zero entries belong to $\{\pm 1\}$. It turns out that the spectrum of the signed adjacency matrix is invariant under switching isomorphism:
 there are all the requisites for the development of a \emph{spectral theory of signed graphs} \cite{zasmat}.
 \\\indent Spectral theory of signed graphs is not only a collection of  extensions of  results from the classical setting. One of its first results, for example, is the Acharya's characterization  of  balance  \cite{acharya}, and it is peculiar of signed graphs. It states that
 a signed graph is  balanced if and only if it is  cospectral with its underlying graph. Moreover, also direct generalizations from classical theory are often far from being trivial, and many questions remain still open \cite{open}.
Many works on the subject can be found in the recent literature, even restricting the field to cospectrality \cite{ak,BB,cos, lolli}. In particular in \cite{cos} the Godsil-McKay switching is extended to signed graphs.
 \\\indent A further generalization of the notion of graph is that of  gain graph, or also voltage graphs. For a given group $G$, a  $G$-gain graph (or gain graph over $G$) is a pair $(\Gamma,\psi)$ where $\Gamma$ is the underlying graph and $\psi$ is the gain function, that is  a map from the pairs $(u,v)$  of adjacent vertices $u,v$ of $\Gamma$ to the group $G$, with the property that $\psi(u,v)=\psi(v,u)^{-1}$, so that $\psi(u,v)$ is the gain from $u$ to $v$. The reason behind the introduction of  gain graphs is to be found, on the one hand in the theory of biased graphs, of which the gain graphs are special cases \cite{zaslavsky1}, on the other hand in the theory of coverings in topological graph theory \cite{voltage}. For the numerous interconnections with other fields and for a rich and regularly updated bibliography, we refer to \cite{zasbib}. Anyway, it is evident how the concept of balance is immediately generalized from signed graphs to gain graphs, as well as, perhaps less immediately, that of switching isomorphism.
 \\\indent When the group $G$ is a subgroup of the  multiplicative group of the complex numbers $\mathbb C$, there is a natural complex matrix playing the role of adjacency matrix of a gain graph $(\Gamma, \psi)$. Moreover, if $G$ is a subgroup of the group $\mathbb T=\{z\in \mathbb{C} : |z|=1\}$ of the complex units, this matrix is also Hermitian with a real spectrum, which is also  a switching isomorphism invariant.
This allowed a development  of a \emph{spectral theory of $\mathbb T$-gain graphs} \cite{reff1}. In particular  Acharya's characterization  of  balance  \cite{adun} and Godsil-McKay switching \cite{gm} have been extended to $\mathbb T$-gain graphs.
\\\indent Some other special groups $G$ have been considered as group of gains for a graph. For example the group of invertible elements of a finite field in \cite{shahulgermina} and the group of quaternion units in
\cite{quat}. In both cases, there was already a definition of spectrum for matrices with entries in $G\cup\{0\}$ and then a definition for the spectrum of a $G$-gain graph. Among other things, in \cite{quat,shahulgermina} it has been proven an extension of Acharya's characterization to these gain graphs.
\\\indent The main obstruction for a spectral theory of   general gain graphs, without any assumption on the group of gains, is that the adjacency matrices are not complex valued. Possible solutions are offered by the work \cite{JACO}, where the generalization of  Acharya's characterization is given in terms of \emph{group algebra valued matrices} $M_n(\mathbb C G)$ (see \cite[Theorem 4.2]{JACO}) and in terms of \emph{$\pi$-spectrum}, or \emph{spectrum of a gain graph with respect to a representation} $\pi$ of the group $G$ of gains (\cite[Theorem 5.1]{JACO}).
In both cases, the balance of $(\Gamma,\psi)$ is  characterized via some kind of cospectrality with $(\Gamma,\bold{1_G})$, that is the gain graph whose underlying graph is $\Gamma$ and the gain function is constantly $1_G$, the unit element of $G$.
\\\indent This manuscript is conceptually the continuation of \cite{JACO}, aiming  at investigating cospectrality  of  $(\Gamma,\psi)$ with any other gain graph, and not only with $(\Gamma,\bold{1_G})$ as in \cite{JACO}. This generalization is not trivial and there are several unexpected results, especially in the group representation approach.
  \\\indent As we said, spectral graph theory is the investigation of  graphs via linear algebra, watching them like operators on a vector space.
 This is exactly what group representation theory does with groups and group elements. It seems very natural to combine these two approaches  in order to develop a spectral theory of gain graphs. Once a (unitary) representation $\pi$ of $G$ is fixed, every $G$-gain graph $(\Gamma,\psi)$ has a complex (Hermitian) adjacency matrix,
 obtained via  an extension  to $M_n(\mathbb C G)$ of the Fourier transform at $\pi$. And then $\pi$-spectrum and $\pi$-cospectrality  are  defined.
 Moreover,  the signed and complex unit spectral investigations we have talked about are covered as particular cases with a suitable choice of the representation $\pi$
 (see \Cref{ex:2,ex:3}).
\\\indent  When $\pi$ is faithful (e.g., when it is the left regular representation $\lambda_G$), we know from \cite{JACO}  that if a gain graph $(\Gamma,\psi)$ is $\pi$-cospectral with $(\Gamma,\bold{1_G})$ then it is balanced and
 switching isomorphic with $(\Gamma,\bold{1_G})$. In particular $(\Gamma,\psi)$ and $(\Gamma,\bold{1_G})$ are cospectral with respect to any other representation.   On the contrary two gain graphs $(\Gamma,\psi)$ and $(\Gamma',\psi')$ can be cospectral with respect to a faithful representation even if they are not switching isomorphic and even if they are not cospectral with respect to some other (faithful, irreducible) representation (see \Cref{exa:2}).
 For finite cyclic groups $\mathbb T_m$ it is at least true that cospectrality with respect to a faithful irreducible representation is equivalent to cospectrality with respect to any other faithful irreducible representation (see \Cref{coro:ultimo}). But this is false in general (see \Cref{exa:s4}). This gives rise to a need for a more canonical definition of cospectrality
 that was not evident after \cite{JACO}, that is one of the main motivations of this paper, leading to the introduction of the notion of $G$-cospectrality.

 \noindent The paper is organized as follows.
 \\ \indent In \Cref{sec:2} we provide the  essential tools on group representations and their extension to the algebras $\mathbb C G$ and  $M_n(\mathbb C G).$
 \\ \indent
In   \Cref{sec:3} we give some preliminaries on gain graphs. We characterize the switching equivalence class of a gain graph in terms of \emph{simultaneous conjugation} of the gains of its closed walks centered at a vertex (see \Cref{thm:primo}) and  we prove that switching equivalence classes of gain graphs on a graph $\Gamma$ with $n$ vertices and $m$ edges are as many as the simultaneous conjugacy classes of the group $G^{m-n+1}$ (see \Cref{coro:card}). We believe that this is a result whose interest can even go beyond the spectral theory.
  \\ \indent In \Cref{sec:4}  we introduce the notion of $G$-cospectrality in  $M_n(\mathbb C G)$ (see \Cref{def:cosp}),  inspired from the trace characterization of cospectrality in $M_n(\mathbb C)$, but in such a way that this definition is switching isomorphism invariant (see \Cref{teo:benposto}).
We prove in \Cref{teo:cosp} that when $G$ is finite, two gain graphs $(\Gamma,\psi)$ and $(\Gamma',\psi')$ are $G$-cospectral if and only if they are cospectral with respect to
all unitary representations of $G$, or equivalently, with respect to a complete system of unitary irreducible representations.
In order to prove this last result, we show how actually all information about the $\pi$-spectrum of a gain graph is on the characters of the gains of the closed walks (see \Cref{teo:pico}).
\\ \indent In addition to the examples already mentioned, in \Cref{sec:5} we analyze, in the light of these new results,  the cases of signed graphs, $\mathbb T_m$-gain graphs and cyclic gain graphs. In particular
we provide non-trivial pairs of $G$-cospectral non-switching isomorphic gain graphs (see  \Cref{exa:1,exa:nuovofine}), and, on the other extreme,
we prove  in \Cref{coro:det}  that every gain graph with cyclic underlying graph is \emph{determined by its $G$-spectrum}. We did this by completely determining the switching equivalence and the switching isomorphism classes  on a cyclic graph. \\
\indent The results of this paper allow to summarize the relationships between different cospectrality notions for two gain graphs $(\Gamma_1,\psi_1)$ and $(\Gamma_2, \psi_2)$ as in Fig. \ref{tableau}.

\begin{figure}

\begin{center}

\begin{tikzcd}[arrows=Rightarrow, column sep=1.3cm, row sep=1.3cm, every arrow/.append style={shift left=0.8ex}]
    &&  \substack{ \text{\bf $\pi$-cospectrality}\\ \text{ $\forall \pi$ in a complete system}\\ \text{ of unitary irreducible reps.} }
        \arrow{d}
    &\\
\substack{ \text{Switching}\\ \text{ Equivalence} }
    \arrow{r}
    & \substack{ \text{Switching}\\ \text{ Isomorphism} }
        \arrow[negated]{l}
            \arrow{r}
            &\text{\bf $G$-cospectrality}
            \arrow{d}
               \arrow{u}
              \arrow[negated]{l}
                  \arrow{r}
             &\text{\footnotesize \bf $\lambda_G$-cospectrality}
             \arrow[negated]{l} \\
             &&  \substack{ \text{\bf $\pi$-cospectrality}\\ \text{ $\forall \pi$ unitary rep. } }
              \arrow{u}
              \arrow{d}
             &\\
             && \substack{ \text{\bf $\pi$-cospectrality}\\ \text{ $\forall \pi$ unitary, irreducible, faithful rep.} }
              \arrow[negated]{u}
              \arrow{d}
              &\\
              && \substack{ \text{\bf $\pi$-cospectrality}\\ \text{ $\exists \pi$ unitary, irreducible, faithful rep.} }
               \arrow[negated]{u}
               \arrow[gray,shift left=-5ex,"\mbox{\tiny if }G=\mathbb T_m"']{u}

           &
             \end{tikzcd}
\end{center}\caption{Cospectrality properties for two $G$-gain graphs $(\Gamma_1,\psi_1)$ and $(\Gamma_2, \psi_2)$ and their connections, with $G$ finite.}\label{tableau}
\end{figure}
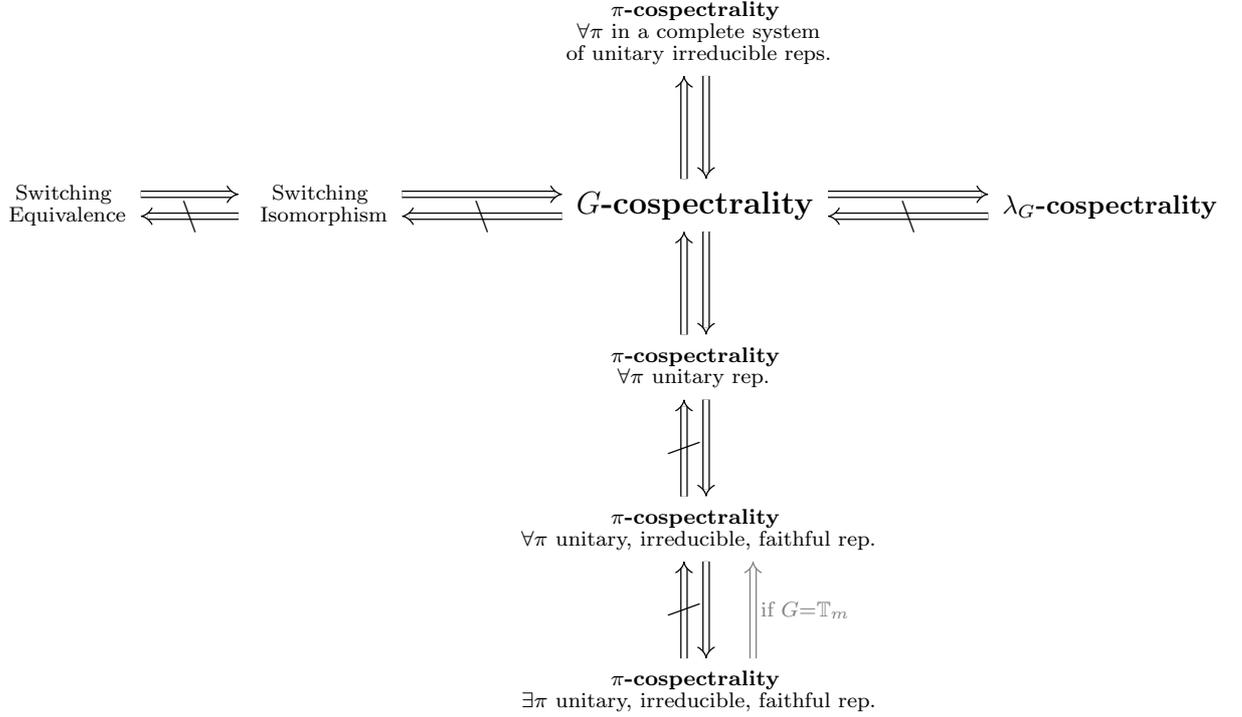

 \section{Preliminaries on representations}\label{sec:2}
Let $G$ be a (finite or infinite) group, with unit element $1_G$. Let $M_k(\mathbb{C})$ be the set of all square matrices of size $k$ with entries in $\mathbb C$
and let $GL_k(\mathbb{C})$ be the group of all invertible matrices in $M_k(\mathbb{C})$. A representation $\pi$ of \emph{degree $k$} (we write $\deg \pi=k$) of $G$  is a group homomorphism $\pi\colon G\to GL_k(\mathbb{C})$.
In other words, we are looking at the elements of $G$ as automorphisms  of a vector space $V$ on $\mathbb C$ with  $\dim V=k$.\\
\indent  Let $U_k(\mathbb{C}) = \{M\in GL_k(\mathbb{C}) : M^{-1}=M^\ast\}$ be the set of unitary matrices of size $k$, where $M^*$ is the Hermitian (or conjugate) transpose of $M$. Then the representation $\pi$ is said to be \emph{unitary} if $\pi(g)\in U_k(\mathbb{C})$, for each $g\in G$. Two representations $\pi$ and $\pi'$ of degree $k$ of a group $G$ are said to be equivalent, and we write $\pi\sim\pi'$, if there exists a matrix $S\in GL_k(\mathbb{C})$ such that, for any $g\in G$, one has $\pi'(g)=S^{-1}\pi(g)S$. This is an equivalence relation and it is known that, for a finite group $G$, each class contains a unitary representative: this is the reason why working with unitary representations is not restrictive in the finite case. \\
\indent We will denote by $I_k$  the identity matrix of size $k$.
 The \textit{kernel} of a representation $\pi$ is $\ker \pi=\{g\in G:\, \pi(g)=I_k\}$, and a representation is said to be \textit{faithful} if $\ker \pi=\{1_G\}$.
The \textit{character} $\chi_\pi$ of a representation $\pi$ is the map $\chi_{\pi}:G\to \mathbb{C}$ defined by $\chi_\pi(g)=Tr(\pi(g))$, that is, the trace of the matrix $\pi(g)$. It is a well known fact that two representations have the same character if and only if they are equivalent.
\\
\indent
Given two representations $\pi_1$ and $\pi_2$ of $G$,  one can construct \emph{the direct sum representation $\pi=\pi_1 \oplus \pi_2$ of $G$},  defined by $\pi(g):=\pi_1(g)\oplus \pi_2(g)$, for every $g\in G$, where $\pi_1(g)\oplus \pi_2(g)$ is the direct sum of the matrices $\pi_1(g)$ and $\pi_2(g)$. We will use the notation $\pi^{\oplus i}$ for the $i$-th iterated direct sum of $\pi$ with itself.
\\
\indent A representation $\pi$ of $G$ is said to be \textit{irreducible} if there is no non-trivial invariant subspace for the action of $G$. It is proven that $\pi$ is irreducible if and only if it is not equivalent to any direct sum of representations.  It is well known that if $G$ has $m$ conjugacy classes, then there exists a list of $m$ irreducible, pairwise  inequivalent, representations $\pi_0,\ldots,\pi_{m-1}$, called a \emph{complete system of irreducible representations of $G$}. Moreover, for every representation $\pi$ of $G$, there exist $k_0,\ldots,k_{m-1} \in \mathbb N \cup \{0\}$ such that
\begin{equation*}\label{deco}
\pi\sim \bigoplus_{i=0}^{m-1} \pi_i^{\oplus k_i}\quad \mbox{ or, equivalently, } \quad \chi_\pi=\sum_{i=0}^{m-1} k_i \chi_{\pi_i}.
\end{equation*}
We say that $\pi$ \emph{contains} the irreducible representation $\pi_i$ with multiplicity $k_i$  if  $k_i\neq 0$, or also that $\pi_i$ is a \emph{subrepresentation} of $\pi$.
\\
\indent
Now we recall some remarkable representations of $G$. The first one is the trivial representation $\pi_0\colon G\to \mathbb C$, with $\pi_0(g)=1$ for every $g\in G$. It is unitary and irreducible and it  always belongs to  a complete system of unitary irreducible representations of $G$. \\
\indent The group $G$  naturally acts by left multiplication on itself. This action can be regarded as the action of $G$   on the vector space
$\mathbb{C}G=\{\sum_{x\in G} c_x x: \ c_x\in \mathbb{C}\}$. When $G$ is finite, this gives rise to the \emph{left regular representation} $\lambda_G$, which is faithful and has degree $|G|$.
It is well known that  $\lambda_G$ contains, up to equivalence, each irreducible representation $\pi_i$ of $G$ with multiplicity $\deg \pi_i$ and its character is non-zero only on $1_G$. In formulae, we have:
\begin{equation*}\label{regdec}
\lambda_G \sim \bigoplus_{i=0}^{m-1} \pi_i^{\oplus \deg \pi_i}\quad \mbox{ and } \quad \chi_{\lambda_G}(g)=
\begin{cases}
|G| &\mbox{ if } g=1_G\\
0 &\mbox{ otherwise.}
\end{cases}
\end{equation*}
For further information and a more general discussion on group representation theory, we refer the interested reader to \cite{fulton}.

Even when $G$ is not finite, the space of all finite $\mathbb C$-linear combinations of elements of $G$ is a vector space, denoted by $\mathbb{C}G$.
Endowed with the product
$$
\left(\sum_{x\in G} f_x x\right) \left(\sum_{y\in G} h_y y\right):= \sum_{x,y\in G} f_x h_y \, x y, \qquad \mbox{ for each } f=\sum_{x\in G} f_x x,\;h=\sum_{y\in G} h_y y\in \mathbb CG,
$$
and with the involution
$$
f^*:= \sum_{x\in G} \overline{f_{x^{-1}} }x,
$$
it is an algebra with involution, known as \emph{group algebra} of $G$.

A representation $\pi$ of degree $k$ of $G$ can be naturally extended by linearity to $\mathbb{C}G$. With a little abuse of notation, we denote with $\pi$ this extension, that is in fact a homomorphism of algebras:
\begin{equation*}
\begin{split}
 \pi\colon \mathbb{C}G&\to M_k(\mathbb C)\\
\sum_{x\in G} f_x x&\mapsto \sum_{x\in G} f_x \pi(x).
\end{split}
\end{equation*}
Moreover, if $\pi$ is unitary, then $\pi(f^*)=\pi(f)^*$ for any $f\in\mathbb CG$.

Also the space $M_n(\mathbb{C}G)$ of the group algebra valued matrices $n\times n$ is an algebra with involution. An element  $F \in M_n(\mathbb C G)$ is in fact a square matrix of size $n$ whose entry $F_{i,j}$ is an element of $\mathbb C G$. The product  of $F,H\in M_n(\mathbb{C}G)$ is such that $(F H)_{i,j}=\sum_{k=1}^n F_{i,k} H_{k,j},$
where $F_{i,k} H_{k,j}$ is the product in the group algebra $\mathbb C G$, and the involution $^*$ in $M_n(\mathbb C G)$ is defined by $(F^*)_{i,j}=(F_{j,i})^*$, where the $^*$ on the right is the involution in $\mathbb C G$.

 A representation $\pi$ of degree $k$ of $G$  can be also extended by linearity  to $M_n(\mathbb{C}G)$:
 $$\pi\colon M_n(\mathbb{C}G)\to M_{nk}(\mathbb C).$$
 For $A\in M_n(\mathbb{C}G)$, the element $\pi(A)\in M_{nk}(\mathbb C)$ is called the \emph{Fourier transform of $A$ at $\pi$}.
 Roughly speaking, $\pi(A)$ is the matrix obtained from the matrix  $A\in M_{n}(\mathbb C G)$ by replacing each occurrence of $g\in G$ with the block $\pi(g)$ and each $0\in \mathbb C G$ with a zero block of size $k$. One can check that also this extension of $\pi$ is a homomorphism of algebras (see, e.g., \cite{JACO}).
 Moreover, equivalence relations and subrepresentations are also preserved in their extensions to $M_n(\mathbb{C}G)$ as the next proposition claims.
\begin{proposition}\label{productfou}{\cite[Proposition 3.4]{JACO}}
Let $A \in M_n(\mathbb C G)$ and let $\pi,\pi'$ be two representations. Then:
\begin{itemize}
\item if $\pi \sim \pi'$ then the matrices $\pi(A)$ and $\pi'(A)$ are similar;
\item the matrix $(\pi\oplus\pi')(A)$ is similar to the matrix $\pi(A)\oplus \pi'(A)$.
\end{itemize}
Moreover, if $\pi$ is unitary, one has
$\pi(A^*)=\pi(A)^*.$
\end{proposition}

\section{Gain graphs and switching equivalence}\label{sec:3}

Let $\Gamma=(V_\Gamma,E_\Gamma)$ be a finite simple graph. If $e=\{u,v\}\in E_\Gamma$   we write $u\sim v$ and we say that $u$ and $v$ are \emph{adjacent vertices} of $\Gamma$.
 Let $G$ be a group and consider a map $\psi$ from the set of ordered pairs of adjacent vertices of $\Gamma$ to $G$,
 such that $\psi(u,v)=\psi(v,u)^{-1}$. The pair $(\Gamma,\psi)$ is said to be a \emph{$G$-gain graph}  (or equivalently, a gain graph on $G$): the graph $\Gamma$ is  called  \emph{underlying graph} and $\psi$ is called a \emph{gain function}  on $\Gamma$.
 \\When $G=\mathbb T_2=\{\pm1\}$, then a $\mathbb T_2$-gain graph, usually denoted by $(\Gamma,\sigma)$, is  called a \emph{signed graph} and the gain function $\sigma$ is  called  a \emph{signature} on $\Gamma$.
 \\Let us fix an ordering $v_1,v_2,\ldots, v_n$ in $V_\Gamma$. The adjacency matrix  of a gain graph  $(\Gamma,\psi)$ is the group algebra valued  matrix $A_{(\Gamma,\psi)}\in M_n(\mathbb C G)$ with entries
$$
(A_{(\Gamma,\Psi)})_{i,j}=
 \begin{cases}
 \psi(v_i,v_j) &\mbox{if } v_i\sim v_j;
  \\  0 &\mbox{otherwise.}
\end{cases}
$$
Notice that
$$
A_{(\Gamma,\psi)}^*=A_{(\Gamma,\psi)}.
$$
Let $W$ be a \emph{walk} of length $h$, that is, an ordered sequence of $h+1$  (not necessarily distinct) vertices of $\Gamma$, say $w_0,w_1,\ldots, w_h$, with $w_i\sim w_{i+1}$. One can define  $\psi(W)\in G$, the \emph{gain of the walk $W$}, as follows:
$$
\psi(W):=\psi(w_0,w_1)\cdots \psi(w_{h-1},w_h).
$$
A \emph{closed walk} of length $h$ is a walk of length $h$ with $w_0=w_{h}$.
In what follows, we denote by $\mathcal W^h_{i,j}(\Gamma)$ the set of walks in $\Gamma$  of length $h$ from $v_i$ to $v_j$, and by $\mathcal C^h(\Gamma):=\bigcup_{i=1}^{n} \mathcal W_{i,i}^h$ the set of all closed walks in $\Gamma$ of length $h$. Notice that
for $W_1\in\mathcal W^{h_1}_{i,j}(\Gamma)$ and $W_2\in\mathcal W^{h_2}_{j,k}(\Gamma)$ one can define the \emph{concatenation} $W_1W_2\in \mathcal W^{h_1+h_2}_{i,k}$ in the obvious way, satisfying the property $\psi(W_1W_2)=\psi(W_1)\psi(W_2)$.
Moreover, for any vertex $v\in V_\Gamma$, we denote by $\mathcal C_v$ the set of  all closed walks starting and ending at the vertex $v$, or  closed walks \emph{centered} at $v$ for short.

A gain graph $(\Gamma,\psi)$ is \emph{balanced} if $\Psi(W)=1_G$ for every closed walk $W$.
An example of a balanced gain graph is that endowed with the \emph{trivial gain function $\bold{1}_G$}, defined as $\bold{1}_G(u,v)=1_G$ for every pair of adjacent vertices $u$ and $v$.
When the underlying graph $\Gamma$ is a tree, the gain graph $(\Gamma,\psi)$ is 	automatically balanced.
\\A fundamental concept in the theory of gain graphs, inherited from the theory of signed graphs, is the \emph{switching equivalence}: two gain functions $\psi_1$ and $\psi_2$ on the same underlying graph
$\Gamma$ are switching equivalent, and we shortly write $(\Gamma,\psi_1)\sim(\Gamma,\psi_2)$, if there exists $f\colon V_\Gamma\to G$ such that
\begin{equation}\label{eqsw}
\psi_2(v_i,v_j)=f(v_i)^{-1} \psi_1(v_i,v_j)f(v_j).
\end{equation}
for any pair of adjacent vertices $v_i$ and $v_j$ of $\Gamma$. If Eq.~\eqref{eqsw} holds, we simply write $\psi_2=\psi_1^f$. We denote by $[\psi]$ the switching equivalence class of the $G$-gain function $\psi$.

It turns out that a gain graph $(\Gamma, \Psi)$ is balanced if and only if $(\Gamma, \Psi)\sim(\Gamma, \bold{1_G})$ (see \cite[Lemma~5.3]{zaslavsky1} or \cite[Lemma~1.1]{refft}). Moreover, in analogy with the complex case (see \cite[Lemma~3.2]{reff1}),
the following result holds.

\begin{theorem}{\cite[Theorem 4.1]{JACO}}\label{teo-vecchio}
Let $\psi_1$ and $\psi_2$ be two gain functions on the same underlying graph $\Gamma$, with adjacency matrices $A_1$ and $A_2\in M_n(\mathbb C G)$, respectively.
Then $(\Gamma,\psi_1)\sim(\Gamma,\psi_2)$ if and only if there exists a diagonal matrix $F\in M_n(\mathbb C G)$, with $F_{i,i}\in G$ for each $i=1,\ldots, n$, such that $F^*A_1F=A_2$.
\end{theorem}
It is worth mentioning that it is also possible  to introduce the notion of a \emph{group algebra valued Laplacian matrix of $(\Gamma,\psi)$}  obtaining the analogous result \cite[Theorem 3.5]{GLine}.
In particular, Theorem \ref{teo-vecchio} implies that the group algebra valued adjacency matrices of switching equivalent gain graphs are, in a way, similar.

Two gain graphs $(\Gamma_1,\psi_1)$ and $(\Gamma_2,\psi_2)$ are said to be \emph{switching isomorphic} if there is a graph isomorphism $\phi\colon V_{\Gamma_1}\to V_{\Gamma_2}$ such that
$(\Gamma_1,\psi_1)\sim(\Gamma_1,\psi_2\circ \phi)$, where $\psi_2\circ \phi$ is the gain function on $\Gamma_1$ such that $(\psi_2\circ \phi)(u,v)=\psi_2(\phi(u),\phi(v))$.
\begin{remark}\label{sw-isomorphism}
It is easy to see  that also the adjacency matrices of switching isomorphic gain graphs can be obtained from each other by conjugating with a suitable group algebra valued matrix.
More precisely, two gain graphs $(\Gamma_1,\psi_1)$ and $(\Gamma_2,\psi_2)$ are switching isomorphic if and only if
\begin{equation*}\label{eq:swiso}
(PF)^*A_{(\Gamma_1,\psi_1)}(PF)=A_{(\Gamma_2,\psi_2)},
\end{equation*}
for some diagonal matrix $F\in M_n(\mathbb C G)$, with $F_{i,i}\in G$ for each $i=1,\ldots, n$, and for some matrix $P\in M_n(\mathbb C G)$ whose non-zero entries are equal to $1_G$ and where the  positions of the non-zero entries correspond to those of a permutation matrix.
\end{remark}

We have given the definitions of switching equivalence and switching isomorphism and their matrix characterizations.
Now we are going to present a different description of switching equivalence in terms of the gains of  closed walks, which will be useful for the spectral investigations that we are going to develop in the next sections.
\\\indent It is known \cite[Section~2]{reff1} that, when $G$ is Abelian,
  it is possible to establish whether
$(\Gamma,\psi_1)$ and $(\Gamma, \psi_2)$
are switching equivalent just by
 looking at the  gain of the closed walks. We extend this characterization also to the non-Abelian case. For two given elements $g,h\in G$, we set $g^h = h^{-1}gh$.
\begin{theorem}\label{thm:primo}
Let $(\Gamma, \psi_1)$ and $(\Gamma, \psi_2)$ be two $G$-gain graphs on the same connected underlying graph $\Gamma$. The following are equivalent.
\begin{enumerate}[(i)]
\item $(\Gamma, \psi_1)\sim (\Gamma, \psi_2)$;
\item $\exists v\in V_\Gamma$, $\exists g_v\in G$ such that $\psi_1(W)=\psi_2(W)^{g_v}$ for all $W\in \mathcal C_v$;
\item $\forall v\in V_\Gamma$, $\exists g_v\in G$ such that $\psi_1(W)=\psi_2(W)^{g_v}$ for all $W\in \mathcal C_v$.
\end{enumerate}
\end{theorem}
\begin{proof}
(i)$\implies$(iii)\\
Suppose that $\psi_1=\psi_2^f$ for some $f\colon V_\Gamma\to G$. For every $v\in V_\Gamma$, and for any closed walk $W\in \mathcal C_v$, with ordered vertices $v, w_1,w_2,\ldots, w_k,v$, we have:
\begin{equation}\label{eq:co}
\begin{split}
\psi_1(W)&=\psi_1(v,w_1)\psi_1(w_1,w_2)\cdots \psi_1(w_k,v)\\
&=f(v)^{-1}\psi_2(v,w_1)f(w_1)f(w_1)^{-1}\psi_2(w_1,w_2)f(w_2)f(w_2)^{-1}\cdots \psi_2(w_k,v)f(v)
\\&={f(v)}^{-1}\psi_2(v,w_1)\psi_2(w_1,w_2)\cdots \psi_2(w_k,v)f(v)
\\&=\psi_2(W)^{f(v)},
\end{split}
\end{equation}
and then setting $g_v:=f(v)$ the thesis follows.
\\
(iii)$\implies$(ii)\\
It is obvious.\\
(ii)$\implies$(i)\\
Let us fix a spanning tree $T$ of $\Gamma$  and denote by $E_T$ the subset of $E_\Gamma$
consisting of the edges of $T$. By hypothesis,  there exist $v\in V_\Gamma$ and $g_{v}\in G$  such that
$\psi_1(W)=\psi_2(W)^{g_{v}}$ for all $W\in \mathcal C_{v}$.
Since a gain graph on a tree is balanced, we can find $p,q\colon V_\Gamma\to G$ such that
\begin{equation}\label{eq:1}
\psi'_1(v_i,v_j)=\psi'_2(v_i,v_j)=1_G\quad \mbox{ if } \{v_i,v_j\}\in E_T,
\end{equation}
where $\psi'_1=\psi_1^p$  and $\psi'_2=\psi_2^q$. This construction comes from a standard technique in gain graphs  (see for example  \cite[Lemma~2.2]{reff1} or \cite[Lemma~2.5]{ELine})
inherited from that for signed graphs.
Moreover, if we set $g'_{v}:=q(v)^{-1}g_{v}p(v)$, performing  computations similar to that of Eq.~\eqref{eq:co}, for any $W\in  \mathcal C_{v}$  we get:
\begin{equation}\label{eq:proof2}
\psi'_1(W)=\psi_1(W)^{p(v)}=\psi_2(W)^{g_{v}p(v)}=\psi'_2(W)^{q(v)^{-1}g_{v}p(v)}=\psi'_2(W)^{g'_{v}}.
\end{equation}
 Notice that, since $T$ is a spanning tree, for any two vertices $v_i$ and $v_j$ adjacent in $\Gamma$ but not in $T$ (that is, $\{v_i,v_j\}\in E_\Gamma \setminus E_T$) there exists a closed walk $W_{i,j}\in \mathcal C_{v}$, visiting $v_i$ first and then $v_j$, whose all edges but  $\{v_i,v_j\}$ are in $E_T$.
By  Eq.~\eqref{eq:1} we have $\psi'_1(W_{i,j})=\psi'_1(v_i,v_j)$ and $\psi'_2(W_{i,j})=\psi'_2(v_i,v_j)$.
As a consequence of Eq.~\eqref{eq:proof2}, we obtain:
\begin{equation}\label{eq:ff}
\psi'_1(v_i,v_j)=\psi'_2(v_i,v_j)^{g'_{v}}.
\end{equation}
But clearly Eq.~\eqref{eq:ff}  trivially holds also when $\{v_i,v_j\}\in E_T$.
Therefore, if we  define $f\colon V_\Gamma\to G$ as $f(w)=g'_{v}$ for any $w\in V_\Gamma$, from Eq.~\eqref{eq:ff}  we have
$$\psi'_1={\psi'_2}^f.$$
As a consequence, $\psi'_1$ and $\psi'_2$ are switching equivalent and so $\psi_1$ and $\psi_2$ are switching equivalent by transitivity.
\end{proof}
In particular, Theorem \ref{thm:primo} says that $\psi_1$ and $\psi_2$ are switching equivalent
if and only if their gains on every closed walk are conjugated in $G$ by an element that only depends on the starting vertex of the walk.

\begin{remark}\label{rem:preco}
Notice that, in the proof of Theorem \ref{thm:primo}, for implication (ii)$\implies$(i) we actually do not need that $\psi_1(W)=\psi_2(W)^{g_v}$ for all $W\in \mathcal C_v$; fixed a spanning tree $T$, it is enough to check this condition on $|E_\Gamma|-|E_T|=|E_\Gamma|-|V_\Gamma|+1$ closed walks centered at $v$, each crossing exactly one distinct edge in $E_\Gamma\setminus E_T$.
\end{remark}
The number $|E_\Gamma|-|V_\Gamma|+1$ is usually called \emph{circuit rank} of $\Gamma$, and we denote it by $cr(\Gamma)$:  it is the cardinality of a cycle basis (see, for example, \cite{Bharary}). This observation allows us to show in Corollary \ref{coro:card} the existence of a bijection between the switching equivalence classes of gain functions on $\Gamma$ and the \emph{simultaneous conjugacy classes} of $G^{cr(\Gamma)}$.
Here $G^{cr(\Gamma)}$ denotes the $cr(\Gamma)$-th iterated Cartesian product of $G$ with itself;
the simultaneous conjugacy class of an element $\left(x_1,\ldots, x_{cr(\Gamma)}\right)\in G^{cr(\Gamma)}$ is the subset
$\left\{\left(g^{-1}x_1g,\ldots, g^{-1}x_{cr(\Gamma)}g\right), g\in G   \right\} \subset G^{cr(\Gamma)}$.

\begin{corollary}\label{coro:card}
Let $\Gamma = (V_\Gamma, E_\Gamma)$ be a connected graph with $n$ vertices and $m$ edges. There is a bijection between
the set of switching equivalence classes of $G$-gain functions on $\Gamma$ and the set of simultaneous conjugacy classes of $G^{m-n+1}$.
\end{corollary}
\begin{proof}
Let us start by fixing a spanning tree $T$ and a vertex $v$  of $\Gamma$. Set $E_\Gamma\setminus E_T=\{e_1,\dots, e_{m-n+1}\}$ and let $W_1,W_2,\ldots, W_{m-n+1}$ be closed walks centered at $v$ such that the only edge in $E_\Gamma\setminus E_T$ crossed by $W_i$  is $e_i$.\\
We first define a map $\phi$ from the  $G$-gain functions on $\Gamma$ to $G^{m-n+1}$:
$$\phi(\psi)=(\psi(W_1),\psi(W_2),\ldots,\psi(W_{m-n+1}) ).$$
The map $\phi$ is surjective. In fact, for every
$h=(h_1,\ldots, h_{m-n+1})\in G^{m-n+1}$ one can define a gain function $\psi_h$ such that $\phi(\psi_h)=h$ as  follows. Let us make
$\psi_h$ be trivial on the pairs of vertices adjacent in $T$. If instead $\{u,v\}\in E_\Gamma\setminus E_T$, then there exists $i\in\{1,\ldots,{m-n+1} \}$ such that $u,v$ or $v,u$ are consecutive vertices in $W_i$. In the first case we set $\psi_h(u,v)=h_i$ (and then $\psi_h(v,u)=h_i^{-1}$), in the second case we set $\psi_h(v,u)=h_i$ (and then $\psi_h(u,v)=h_i^{-1}$). By definition of $\phi$ we have
$\phi(\psi_h)=(h_1,h_2,\ldots,h_{m-n+1})$ and then $\phi$ is surjective.

By virtue of Theorem \ref{thm:primo}  $(i)\implies(ii)$, if $(\Gamma,\psi_1)\sim (\Gamma,\psi_2)$, then $\phi(\psi_1)$ and $\phi(\psi_2)$ are simultaneously conjugated.
Vice versa, if two gain functions $\psi_1$ and $\psi_2$ are such that  $\phi(\psi_1)$ and $\phi(\psi_2)$ are simultaneously conjugated, then there exist $g_v\in G$ such that
$\psi_1(W_i)=\psi_2(W_i)^{g_v}$ for each  $i=1,\ldots, |E_\Gamma|-|V_\Gamma|+1$.
By virtue of Remark \ref{rem:preco} this is enough to ensure that  $(\Gamma,\psi_1)\sim (\Gamma,\psi_2)$.
 As a consequence, the map $\phi$ composed with the projection onto the simultaneous conjugacy class,  induces a bijection between the set of switching equivalence classes of gain functions on $\Gamma$ and the simultaneous conjugacy classes of $G^{m-n+1}$.
\end{proof}

Notice that, when $G$ is finite and Abelian, Corollary \ref{coro:card} implies that for a connected graph $\Gamma$ there are $|G|^{m-n+1}$ switching classes of $G$-gain functions on $\Gamma$. In particular there are $2^{m-n+1}$ switching classes of signed graphs on $\Gamma$ (see \cite[Proposition~3.1]{nase}).

\section{$G$-cospectrality, $\pi$-spectrum and $\pi$-cospectrality}\label{sec:4}
For classical graphs, signed graphs, and $\mathbb T$-gain graphs, there is  one natural  definition of the (adjacency) spectrum, that is in particular invariant under switching isomorphism. In those settings there is no ambiguity, essentially because the adjacency matrices are complex valued matrices.
On the contrary,  several definitions of spectrum may be associated with an element of $M_n(\mathbb C G)$ (and therefore with a $G$-gain graph)   within the framework of operator algebras  \cite{Chu}. As already done in \cite{JACO}, we will deal with the problem bringing us back to the theory of complex matrices. More precisely, as a first step, we define \emph{cospectrality} in $M_n(\mathbb C G)$ without giving an explicit definition for the spectrum in $M_n(\mathbb C G)$.
Only after fixing a unitary representation of $G$, we define the spectrum associated with that representation.

As already said, in order to define cospectrality in $M_n(\mathbb C G)$, we are inspired   from what happens in the classical setting of complex matrices.
The following lemma is a consequence of Specht's theorem  \cite{specht} for Hermitian matrices.
\begin{lemma}\label{specht}
Let $A,B\in\mathbb M_n(\mathbb C)$ be Hermitian. Then
$$A \mbox{ and } B \mbox{ are cospectral } \iff Tr(A^h)=Tr(B^h)\quad \forall h\in \mathbb N.$$
\end{lemma}
Alternatively, elementary symmetric polynomials can be used to prove that, for $A,B\in M_n(\mathbb C)$, one has that $Tr(A^h)=Tr(B^h)$ for $h=1,\ldots,n$ is equivalent to the cospectrality of $A$ and $B$, see \cite{rob}.

Observe that if $A$ is the adjacency matrix of a graph $\Gamma=(V_\Gamma,E_\Gamma)$, then the entry  $(A^h)_{i,j}$ equals the number of walks of length $h$ from $v_i$ to $v_j$. As a consequence, looking at $i=j$, one has:
\begin{equation}\label{eq:nogain}
\Gamma_1 \mbox{ and }  \Gamma_2 \mbox{ are cospectral graphs } \iff |\mathcal C^h(\Gamma_1)|=|\mathcal C^h(\Gamma_2)|\quad \forall h\in \mathbb N.
\end{equation}

Keeping  this  characterization in mind, let us introduce a group algebra valued trace map on $M_n(\mathbb C G)$ as follows:
\begin{equation}\label{eq:trace}
\begin{split}
Tr\colon M_n(\mathbb C G)&\to \mathbb C G\\
A&\mapsto \sum_{i=1}^n A_{i,i}.
\end{split}
\end{equation}
If $A\in  M_n(\mathbb C G)$ is the adjacency matrix of a $G$-gain graph $(\Gamma,\psi)$, by virtue of \cite[Lemma~4.1]{JACO}  the entry $(A^h)_{i,j}\in\mathbb CG$ is the sum of the gains of  all walks from $v_i$ to $v_j$ of length $h$.
In particular, we have:
\begin{equation}\label{eq:trko}
Tr(A^h)=\sum_{W \in\mathcal C^h(\Gamma)} \psi(W).
\end{equation}

We might be tempted to define the cospectrality of  $A,B\in M_n(\mathbb C G)$ by requiring the equality of $Tr(A^h)$ and $Tr(B^h)$ in  $\mathbb C G$ for all $h\in \mathbb N$ and, therefore, to define the cospectrality of two $G$-gain graphs $(\Gamma_1,\psi_1)$ and $(\Gamma_2,\psi_2)$ by requiring the equality of the sum of the gains of  all closed walks of length $h$, for all $h\in \mathbb N$.
However, this request when the group $G$ is not Abelian is  too strong. In fact the  map $Tr\colon M_n(\mathbb C G)\to \mathbb C G$, defined in Eq.~\eqref{eq:trace}, is not invariant under conjugation and there are pairs of switching isomorphic (even switching equivalent) graphs which would not be cospectral in this sense (see Section \ref{sec:cicli}).
What we really want from two cospectral matrices $A,B\in M_n(\mathbb C G)$ is that $Tr(A^h)$ and $Tr(B^h)$ are equal only up to group conjugations of some of their addends.

Let us denote by $[G]$ the set of conjugacy classes of $G$,
and by $[g]$ the conjugacy class of $g\in G$. A \emph{class function} $f\colon G\to \mathbb C$ is a map such that $g_1,g_2\in[g]\implies f(g_1)=f(g_2)$. The set  $\mathbb C_{Class}[G]$ of finitely supported class functions is a $\mathbb C$-vector space.
Moreover, if $G$ is finite, the vector space $\mathbb C_{Class}[G]$ can be endowed with a Hermitian inner product:
\begin{equation}\label{eq:pro}
\begin{split}
&\langle\;,\;\rangle\colon  \mathbb C_{Class}[G]\times  \mathbb C_{Class}[G] \to \mathbb C\\
&\langle f  , h \rangle=\frac{1}{|G|}\sum_{g\in G}  f(g) \overline{h(g)}.\\
\end{split}
\end{equation}

There is a natural map $\mu$ from $\mathbb CG$ to $\mathbb C_{Class}[G]$, defined as  the sum of the coefficients on each conjugacy class:
\begin{equation}\label{eq:projection}
\begin{split}
&\mu\colon \mathbb C G\to \mathbb C_{Class}[G]\\
 &\mu\left(\sum_{x\in G} a_x x\right)(g)=\sum_{x\in[g]} a_x.
\end{split}
\end{equation}
Notice that, if $G$ is Abelian, each conjugacy class contains only one element and $\mu$ is nothing but an isomorphism between $\mathbb C$-vector spaces.

\begin{definition}\label{def:cosp}
Two group algebra valued matrices $A,B\in M_n(\mathbb C G)$, with $A^*=A$ and $B^*=B$, are \emph{$G$-cospectral} if and only if
$$\mu(Tr(A^h))=\mu(Tr(B^h))\quad \forall h\in \mathbb N.$$
Two gain graphs $(\Gamma_1,\psi_1)$ and $(\Gamma_2,\psi_2)$ are \emph{$G$-cospectral} if $A_{(\Gamma_1,\psi_1)}$ and $A_{(\Gamma_2,\psi_2)}$ are $G$-cospectral.
\end{definition}
Our first result consists in reformulating $G$-cospectrality in terms of the gains of closed walks, in analogy with the classical case (see Eq.~\eqref{eq:nogain}).

\begin{proposition}\label{prop:cicli}
Two gain graphs $(\Gamma_1,\psi_1)$ and $(\Gamma_2,\psi_2)$ are $G$-cospectral if and only if
\begin{equation}\label{eq:prop}
\left|\{W\in \mathcal C^h(\Gamma_1): \psi_1(W)\in [g]\} \right|
= \left|\{W\in \mathcal C^h(\Gamma_2): \psi_2(W)\in [g]\} \right|,\quad \forall [g]\in[G],\;\forall h\in\mathbb N.
\end{equation}
\end{proposition}
\begin{proof}
Let $A:=A_{(\Gamma_1,\psi_1)}$ and $B:=A_{(\Gamma_2,\psi_2)}$. By Definition \ref{def:cosp}, the graphs $(\Gamma_1,\psi_1)$ and $(\Gamma_2,\psi_2)$ are $G$-cospectral if and only if
$\mu(Tr(A^h))=\mu(Tr(B^h))$, for all $h\in \mathbb N$. By Eq.~\eqref{eq:trko}, this is  equivalent to
\begin{equation}\label{eq:num}
\mu\left( \sum_{W\in  \mathcal C^h(\Gamma_1)} \psi_1(W) \right)=\mu\left( \sum_{W\in  \mathcal C^h(\Gamma_2)} \psi_2(W) \right),\quad \forall h\in \mathbb N.
\end{equation}
On the other hand, by the definition of $\mu$ in Eq.~\eqref{eq:projection}, for every $g\in G$, we have:
$$\mu\left( \sum_{W\in  \mathcal C^h(\Gamma_i)} \psi_i(W) \right)(g)= \left|\{W\in \mathcal C^h(\Gamma_i): \psi_i(W)\in [g]\} \right|.$$
Therefore Eq.~\eqref{eq:num} is equivalent to Eq.~\eqref{eq:prop} and the thesis follows.
\end{proof}

A famous result of Acharya states that a signed graph is balanced if and only if it is cospectral with its underlying graph \cite{acharya}. There are several generalizations of this result \cite{adun,JACO}. Definition \ref{def:cosp} and Proposition
\ref{prop:cicli} allow us to easily prove the following.
\begin{theorem}
A gain graph $(\Gamma,\psi)$ is balanced if and only if it is $G$-cospectral with  $(\Gamma,\bold{1}_G)$.
\end{theorem}
\begin{proof}
It is clear that all closed walks of length $h$ in $(\Gamma,\bold{1}_G)$ have gain $1_G$, for each $h\in\mathbb N$. Then, by Proposition \ref{prop:cicli},
$(\Gamma,\psi)$ is $G$-cospectral with $(\Gamma,\bold{1}_G)$ if and only if also all closed walks of length $h$ in $(\Gamma,\psi)$ have gain $1_G$, for each $h\in\mathbb N$, that is, if and only if $(\Gamma,\psi)$ is balanced.
\end{proof}

Combining Theorem \ref{thm:primo} with Proposition \ref{prop:cicli}, we deduce that if $(\Gamma,\psi_1)\sim(\Gamma,\psi_2)$, then $(\Gamma,\psi_1)$ and $(\Gamma,\psi_2)$ are $G$-cospectral. Actually, the following result guarantees that  the property of $G$-cospectrality  is invariant under switching isomorphism.

\begin{theorem}\label{teo:benposto}
If two $G$-gain graphs $(\Gamma_1,\psi_1)$ and $(\Gamma_2,\psi_2)$ are switching isomorphic, then they are
$G$-cospectral.
\end{theorem}
\begin{proof}
The isomorphism $\phi$ from $\Gamma_1$ to $\Gamma_2$ induces a bijection $\phi$ from $\mathcal C^h(\Gamma_1)$ to  $\mathcal C^h(\Gamma_2)$, for all $h\in\mathbb N$. In particular, if $W$ is a closed walk of length $h$ in $\Gamma_1$, $\phi(W)$ is a closed walk of length $h$ in $\Gamma_2$.
Moreover, by definition of switching isomorphism,
$(\Gamma_1,\psi_1)\sim(\Gamma_1,\psi_2\circ \phi)$
and then, by Theorem \ref{thm:primo}, the elements $\psi_1(W)$ and $\psi_2(\phi(W))$ belong to the same conjugacy class in $G$. This implies that the number of closed walks  in $\Gamma_1$ of length $h$  whose gain is in $[g]$ is equal to the number of closed walks  in $\Gamma_2$ of length $h$  whose gain is in $[g]$, for all $h\in\mathbb N$ and all $[g]\in[G]$. The thesis follows from Proposition \ref{prop:cicli}.
\end{proof}
As one would expect, the converse of Theorem \ref{teo:benposto} is not true: there are pairs of $G$-gain graphs that are $G$-cospectral but  non-switching isomorphic.
One can construct such pairs by considering two non-isomorphic cospectral  graphs $\Gamma_1$ and $\Gamma_2$. Clearly $(\Gamma_1, \bold{1_G})$ and $(\Gamma_2, \bold{1_G})$ are  $G$-cospectral but  non-switching isomorphic  (see  \Cref{exa:1,exa:nuovofine}  for  non-trivial examples).
On the other extreme, there exist  $G$-gain graphs determined by their $G$-spectrum (see Corollary \ref{coro:det}), according to the following definition.

\begin{definition}
A $G$-gain graph $(\Gamma,\psi)$ is \emph{determined by its $G$-spectrum} if
$$(\Gamma,\psi)\mbox{ and }(\Gamma',\psi')  \mbox{ $G$-cospectral }  \implies (\Gamma,\psi) \mbox{ and } (\Gamma',\psi') \mbox{  switching isomorphic.}$$
\end{definition}
In the rest of this section, we recall the notion of spectrum with respect a unitary representation from \cite{JACO} and investigate its relation with  $G$-cospectrality.

For a representation $\pi$  of $G$  of degree $k$ we denote by $\chi_\pi\colon G\to \mathbb C$ its  character. Notice that $\chi_\pi(1_G)= \deg \pi = k$. As already mentioned, the representation
$\pi$ can be extended to $\mathbb C G$ and to $M_n(\mathbb C G)$ via Fourier transforms.
\\ \indent Similarly, the character $\chi_\pi$ can be linearly extended  to $\mathbb C G$, and to finitely supported functions in $\mathbb C_{class} [G]$:
\begin{equation}\label{eq:cara}
\begin{split}
\chi_\pi\left(\sum_{x\in G} a_x x\right)&:=Tr\left( \pi\left(\sum_{x\in G} a_x x\right)\right)=\sum_{x\in G} a_x \chi_\pi(x),\qquad \mbox{with }\sum_{x\in G} a_x x\in \mathbb C G;\\
\chi_\pi(h)&:= \sum_{x\in G} h(x) \chi_\pi(x)=\sum_{[g]\in [G]}|[g]|h(g)\chi_\pi(g),\qquad \mbox{with } h\in\mathbb C_{Class} [G].
\end{split}
\end{equation}

Consider a $G$-gain graph $(\Gamma,\psi)$ with adjacency matrix $A\in\mathbb M_n(\mathbb C G)$. The matrix $A_\pi:=\pi(A)\in M_{nk}(\mathbb C)$ is called the \emph{represented adjacency matrix} of
$(\Gamma,\psi)$ with respect to $\pi$.  By Proposition \ref{productfou},  $A_\pi$ is a Hermitian matrix and we say that its spectrum is the \emph{$\pi$-spectrum of $(\Gamma,\psi)$}, denoted with $\sigma_\pi(\Gamma,\psi)$ or also $\sigma_\pi(A)$. Notice that, if $\pi$ is unitary, then the $\pi$-spectrum is real. Moreover, by Proposition \ref{productfou}, if $\pi\sim\pi'$, then $\sigma_\pi(A)=\sigma_{\pi'}(A)$.
\begin{definition}
Let $(\Gamma_1,\psi_1)$ and $(\Gamma_2,\psi_2)$ be two $G$-gain graphs. Let $\pi$ be a unitary representation of $G$.
The graphs $(\Gamma_1,\psi_1)$ and $(\Gamma_2,\psi_2)$ are said to be $\pi$-cospectral if $\sigma_\pi(\Gamma_1,\psi_1)=\sigma_\pi(\Gamma_2,\psi_2)$.
\end{definition}
Although it was formally introduced only in \cite{JACO}, particular cases of $\pi$-spectra  have often been considered in the literature, as shown in the next examples.
\begin{example}\label{ex:1}
Let $(\Gamma,\psi)$ be a $G$-gain graph and let $\pi_0$ be the trivial representation of $G$, that is, the one-dimensional representation such that
$\pi_0(g)=1$ for any $g\in G$. Then $A_{\pi_0}$ is nothing but the adjacency matrix of the
underlying graph $\Gamma$. Thus the $\pi_0$-spectrum of $(\Gamma,\psi)$ is the spectrum of its underlying graph.
\end{example}

\begin{example}\label{ex:2}
Let $(\Gamma,\sigma)$ be a signed graph and let  $\pi_{id} \colon \mathbb T_2\to \mathbb C$ be the identical one-dimensional representation of $\mathbb T_2$ with values $+1$ and $-1$.
The  represented adjacency matrix $A_{\pi_{id}}$ is the classical adjacency matrix of  the signed graph (see, e.g., \cite{zasmat}) and so  the $\pi_{id}$-spectrum is the classical spectrum of the signed graph.
\end{example}

\begin{example}\label{ex:3}
Let $(\Gamma,\psi)$ be a $\mathbb T$-gain graph and let  $\pi_{id} \colon \mathbb T\to \mathbb C$ the identical one-dimensional representation of $\mathbb T$. The  represented adjacency matrix $A_{\pi_{id}}$ is the classical adjacency matrix for  complex unit gain graphs (in the sense of \cite{reff1}) and  the $\pi_{id}$-spectrum is the classical spectrum of a  complex unit gain graph.
\end{example}

\begin{example}\label{ex:4}
Let $(\Gamma,\psi)$ be a $G$-gain graph, and let $\lambda_G$ be the left regular representation of $G$. Then the  represented adjacency matrix $A_{\lambda_G}$ is the adjacency matrix of the  \emph{(left) cover graph of $(\Gamma,\psi)$} \cite[Lemma~6.1]{JACO}, and the $\lambda_G$-spectrum of $(\Gamma,\psi)$ is the spectrum of the cover graph of $(\Gamma,\psi)$.
\end{example}

 As a first result about \emph{represented cospectrality}, in analogy with Proposition \ref{prop:cicli}, we are going to prove that the $\pi$-spectrum of a gain graph is fully identified by the characters of the gains of the closed walks.

\begin{theorem}\label{teo:pico}
Two gain graphs $(\Gamma_1,\psi_1)$ and $(\Gamma_2,\psi_2)$ are $\pi$-cospectral if and only if
\begin{equation}\label{eq:inteo}
\sum_{W \in\mathcal C^h} \chi_\pi(\psi_1(W))= \sum_{W \in\mathcal C^h} \chi_\pi(\psi_2(W)), \quad \forall h\in\mathbb N.
\end{equation}
\end{theorem}
\begin{proof}
Recall that the extension to $M_n(\mathbb C G)$ of a unitary representation $\pi$ of $G$ of degree $k$ is a homomorphism, and then, for every $A\in M_n(\mathbb C G)$, for any $h\in\mathbb N$, one has:
$$\pi(A^h)=\pi(A)^h.$$
If $(\Gamma,\psi)$ is a $G$-gain graph on $n$ vertices, $A=A_{(\Gamma,\psi)}\in M_n(\mathbb C G)$  is its adjacency matrix and
$A_\pi\in M_{nk}(\mathbb C)$  is its represented adjacency matrix with respect to $\pi$, then by using
Eq.~\eqref{eq:trko} and Eq.~\eqref{eq:cara} one has:
$$Tr\left((A_\pi)^h\right)=Tr\left(\pi\left( A^h\right) \right)=\chi_\pi(A^h)=\sum_{W \in\mathcal C^{h}} \chi_\pi (\psi(W)), \qquad \forall h\in \mathbb N.$$
Combining with Lemma \ref{specht} one can conclude that the represented  adjacency matrices of $(\Gamma_1,\psi_1)$ and $(\Gamma_2,\psi_2)$ are cospectral if and only if Eq.~\eqref{eq:inteo} holds.
\end{proof}
If we apply Theorem \ref{teo:pico} to the trivial representation $\pi_0$ (see Example \ref{ex:1}), we can say that the underlying graphs of two gain graphs are cospectral if and only if they have, for all $h\in\mathbb N$,  the same number of closed walks of length $h$.
For two signed graphs or  two $\mathbb T$-gain graphs with the representation $\pi_{id}$ (see  \Cref{ex:2,ex:3}),  \Cref{teo:pico} implies that they are cospectral if,  for all $h\in\mathbb N$, they share the sum of the gains of their closed walks of length $h$.
Finally, if we apply Theorem \ref{teo:pico} to the left regular representation (see Example \ref{ex:4}), we can say that
$(\Gamma_1,\psi_1)$ and $(\Gamma_2,\psi_2)$ have cospectral cover graphs  if and only if they have,  for all $h\in\mathbb N$, the same number of \emph{balanced closed walks} of length $h$, where a balanced closed walk is a walk whose gain is $1_G$, that is also the only group element whose character is non-zero.

Surprisingly,  $(\Gamma_1,\psi_1)$ and $(\Gamma_2,\psi_2)$ can be $\pi_1$-cospectral but not $\pi_2$-cospectral, even if both unitary representations $\pi_1$ and $\pi_2$ are faithful (see Example \ref{exa:2}), or even if
$\pi_1$ and $\pi_2$ are faithful and irreducible (see Example \ref{exa:s4}). The following proposition shows that,  however, it is  possible to prove $\pi$-cospectrality just by looking at the irreducible subrepresentations of $\pi$.
\begin{proposition}\label{prop:somma}
If two $G$-gain graphs $(\Gamma_1,\psi_1)$ and $(\Gamma_2,\psi_2)$ are $\pi_1$-cospectral and  $\pi_2$-cospectral, where $\pi_1$ and $\pi_2$ are unitary representations, then $(\Gamma_1,\psi_1)$ and $(\Gamma_2,\psi_2)$ are $(\pi_1\oplus \pi_2)$-cospectral.
\end{proposition}
\begin{proof}
By Proposition \ref{productfou} the matrix $(\pi_1\oplus \pi_2)(A)$ is similar to the matrix  $\pi_1(A)\oplus \pi_2(A)$ for any $A\in M_n(\mathbb CG)$. As a consequence, the $(\pi_1\oplus \pi_2)$-spectrum of a gain graph is the union (as multisets) of  the $\pi_1$-spectrum with the $\pi_2$-spectrum. The thesis easily follows.
\end{proof}
Notice that the converse of Proposition \ref{prop:somma} is not true, see Example \ref{exa:2}.
Anyway, as a consequence of Proposition \ref{prop:somma}, if $\pi_0,\ldots,\pi_{m-1}$ is a complete system of irreducible unitary representations of $G$, and $(\Gamma_1,\psi_1)$ and $(\Gamma_2,\psi_2)$  are $\pi_i$-cospectral for each $i=0,\ldots,m-1$, then $(\Gamma_1,\psi_1)$ and $(\Gamma_2,\psi_2)$ are $\pi$-cospectral for every unitary representation $\pi$.

The following theorem ensure that, when $G$ is finite, one can check $G$-cospectrality just by looking at cospectrality with respect to a complete system of irreducible, unitary, representations of $G$.
\begin{theorem}\label{teo:cosp}
Let $(\Gamma_1,\psi_1)$ and $(\Gamma_2,\psi_2)$ be two $G$-gain graphs, with $G$ finite. Let $\pi_0,\ldots,\pi_{m-1}$ be a complete system of irreducible, unitary, representations of $G$. The following are equivalent.
\begin{enumerate}
\item $(\Gamma_1,\psi_1)$ and $(\Gamma_2,\psi_2)$  are $G$-cospectral;
\item $(\Gamma_1,\psi_1)$ and $(\Gamma_2,\psi_2)$  are $\pi$-cospectral, for every unitary representation $\pi$ of $G$;
\item $(\Gamma_1,\psi_1)$ and $(\Gamma_2,\psi_2)$  are $\pi_i$-cospectral, for each $i=0,\ldots,m-1$.
\end{enumerate}
\end{theorem}
\begin{proof}
(1)$\implies$(2)\\
By Proposition \ref{prop:cicli}, if $(\Gamma_1,\psi_1)$ and $(\Gamma_2,\psi_2)$ are $G$-cospectral, then
Eq.~\eqref{eq:prop} holds.
Let $\pi$ be a unitary representation. Since $\chi_{\pi}$ is a class function we have
\begin{equation}\label{eq:dimo2}
\sum_{W \in\mathcal C^h} \chi_\pi(\psi_i(W))=\sum_{[g]\in[G]} \left|\{W\in \mathcal C^h(\Gamma_i): \psi_i(W)\in [g]\} \right| \chi_{\pi}(g), \quad i=1,2;
\;\forall h\in\mathbb N.
\end{equation}
Combining Eq.~\eqref{eq:prop} with Eq.~\eqref{eq:dimo2} we have
$$\sum_{W \in\mathcal C^h} \chi_\pi(\psi_1(W))= \sum_{W \in\mathcal C^h} \chi_\pi(\psi_2(W)), \quad \forall h\in\mathbb N,$$
and  this implies (2) by Theorem \ref{teo:pico}.\\
(2)$\implies$(3)\\It is obvious.\\
(3)$\implies$(1)\\
Let $A,B\in M_n(\mathbb C G)$ be the adjacency matrices of $(\Gamma_1,\psi_1)$ and $(\Gamma_2,\psi_2)$, respectively. Since $A_{\pi_i}$ and $B_{\pi_i}$ are cospectral, by Theorem \ref{teo:pico} we have:
\begin{equation}\label{eq:dimofi2}
\chi_{\pi_i}\left(Tr\left(A^h\right)\right)=\sum_{W \in\mathcal C^{h}} \chi_{\pi_i} (\psi_1(W))= \sum_{W \in\mathcal C^{h}} \chi_{\pi_i}(\psi_2(W))=\chi_{\pi_i}\left(Tr\left(B^h\right)\right), \quad\forall h\in\mathbb N.
\end{equation}
We want to prove that
$\mu(Tr(A^h))=\mu(Tr(B^h))$ for all $h\in\mathbb N$, that is, $(\Gamma_1,\psi_1)$ and $(\Gamma_2,\psi_2)$ are $G$-cospectral. Since $G$ is finite, for any $f\in \mathbb CG$, one can define
$\overline{\mu}(f)\in \mathbb C_{Class}[G]$, that is, the \emph{normalization of $\mu(f)$}, in the following way:
$$\overline{\mu}(f)(g)=\frac{1}{|[g]|} \sum_{x\in[g]} f_x=\frac{1}{|[g]|} \mu(f)(g).$$
One can easily check that $\mu(Tr(A^h))=\mu(Tr(B^h))$ for all $h\in\mathbb N$ if and only if $\overline{\mu}(Tr(A^h))=\overline{\mu}(Tr(B^h))$ for all $h\in\mathbb N$, since both conditions are equivalent to that  of
Eq.~\eqref{eq:prop} in Proposition \ref{prop:cicli}. Moreover, since $G$ is finite, the characters $\chi_{\pi_0}, \chi_{\pi_1}, \ldots, \chi_{\pi_{m-1}}$ form an orthonormal basis of $\mathbb C_{Class}[G]$ with respect to the Hermitian product defined in Eq.~\eqref{eq:pro} (see \cite[Theorem 2.12]{fulton}).
Therefore, in order to conclude the proof, we only need to show that
\begin{equation}\label{eq:dimofinale}
\langle \chi_{\pi_i},\overline{\mu}\left(Tr\left(A^h\right)\right)\rangle
=\langle \chi_{\pi_i},\overline{\mu}\left(Tr\left(B^h\right)\right)\rangle, \quad\forall h\in\mathbb N,\ i=0,1,\ldots, m-1.
\end{equation}
Notice that  the coefficient of $Tr\left(A^h\right)$ and $Tr\left(B^h\right)$ are real, and then
$Tr\left(A^h\right), Tr\left(B^h\right)\in \mathbb R G$.
For $f\in\mathbb R G$, one  has
\begin{equation}\label{eq:ultima}
\langle \chi_{\pi_i},\overline{\mu}(f)\rangle=
\frac{1}{|G|}\sum_{g\in G} \chi_{\pi_i}(g) \overline{\frac{1}{|[g]|}\sum_{x\in [g]} f_x}
=\frac{1}{|G|} \sum_{x\in G} \chi_{\pi_i}(x) f_x= \frac{1}{|G|} \chi_{\pi_i}(f),
\end{equation}
since $\chi_{\pi_i}(x)=\chi_{\pi_i}(g)$ for all $x\in [g]$.
By gluing together Eq.~\eqref{eq:ultima} with Eq.~\eqref{eq:dimofi2} we obtain Eq.~\eqref{eq:dimofinale}, and the thesis follows.
\end{proof}

\begin{remark}\label{rem:nonleft}
As already mentioned, the left regular representation $\lambda_G$ contains every irreducible representations. However the condition for a pair of $G$-gain graphs of being $\lambda_G$-cospectral is not  equivalent to those of Theorem \ref{teo:cosp}, see Example \ref{exa:2}.
\end{remark}

\section{Examples and applications}\label{sec:5}
This section is devoted to the application of the results obtained in the previous sections to some remarkable cases: signed graphs, gain cyclic graphs, $\mathbb{T}_m$-gain graphs. Finally, we discuss a particular example where a pair of gain graphs over the symmetric group $S_4$ is considered.
\subsection{Signed graphs}
Signed graph can be regarded as $\mathbb T_2$-gain graphs, being $\mathbb T_2$ the group of order two. The irreducible representations of $\mathbb T_2$ are the trivial representation $\pi_0$, and the signed representation, that coincides with $\pi_{id}$.

In the light of   \Cref{teo:cosp} and  \Cref{ex:1,ex:2}, two signed graphs are $\mathbb T_2$-cospectral if and only if they are cospectral as signed graphs and they  have cospectral underlying graphs. A trivial example is given by a pair of cospectral  graphs (as unsigned graphs) both endowed with the all positive signature. We present now a non-trivial example.
\begin{example}\label{exa:1}
Let $(K_8,\sigma)$ be the signed graph depicted in Fig. \ref{fig:1}.
An explicit computation shows that its spectrum (that is, its $\pi_{id}$-spectrum) is
$$\{\pm 1,\pm \sqrt{5}^{ (2)},\pm \sqrt{17}\};$$
it is symmetric with respect to $0$. As a consequence,  $(K_8,\sigma)$ and  $(K_8,-\sigma)$ are cospectral as signed graphs, that is, they are $\pi_{id}$-cospectral. Moreover, $(K_8,\sigma)$ and  $(K_8,-\sigma)$  are also $\pi_0$-cospectral since they have isomorphic underlying graphs. By virtue of Theorem \ref{teo:cosp}, $(K_8,\sigma)$ and  $(K_8,-\sigma)$  are $\mathbb T_2$-cospectral, even if they are not switching isomorphic, because $(K_8,\sigma)$ is an example of a non-signsymmetric graph with symmetric spectrum (see \cite[Fig.~2]{open}, \cite{sss}).
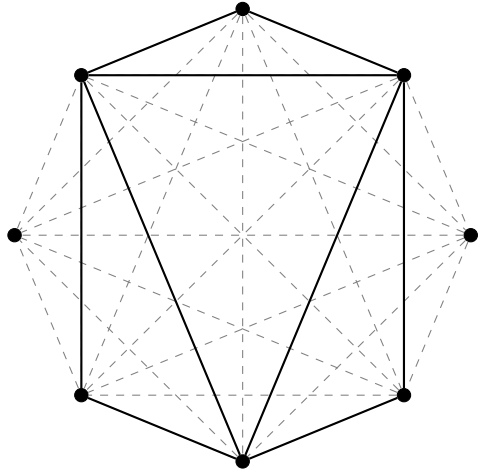
\begin{figure}
\centering
  \begin{tikzpicture}
  \graph { subgraph K_n [n=8,clockwise,empty nodes, radius=3cm,edges={dashed,gray},nodes={draw,circle,inner sep=0pt,minimum size=5pt,fill=black} ]};
    \draw (8) edge[thick] (1);
        \draw (1) edge[thick] (2);
             \draw (2) edge[thick] (4);
                 \draw (4) edge[thick] (5);
                          \draw (5) edge[thick] (6);
                                  \draw (6) edge[thick] (8);
                                           \draw (8) edge[thick] (2);
                                                    \draw (8) edge[thick] (5);
                                                       \draw (2) edge[thick] (5);
\end{tikzpicture}
\caption{The signed graph $(K_8,\sigma)$ of Example \ref{exa:1}. }\label{fig:1}
\end{figure}
\end{example}

\subsection{Cycles}\label{sec:cicli}
We are going to use the results of the previous two sections  to completely characterize  switching equivalence classes, switching isomorphism classes,  $G$-cospectrality classes and $\lambda_G$-cospectrality classes for $G$-gain graphs on the cyclic graph $C_n$ with vertices $v_1,\ldots,v_n$, for any group $G$.

\subsubsection{Switching equivalence classes.} The circuit rank of $C_n$ is  $1$ and then, by virtue of Corollary \ref{coro:card}, the switching equivalence classes of $G$-gain functions on $C_n$ are in bijection with $[G]$, the set of conjugacy classes of $G$. More precisely, fixing $W_0$ the closed walk $v_1,v_2,\ldots,v_n,v_1$ on $C_n$, we have
$$
(C_n,\psi_1)\sim (C_n,\psi_2) \iff [\psi_1(W_0)]=[\psi_2(W_0)].
$$

\subsubsection{Switching isomorphism classes.}
We have just said that on $C_n$ there are $|[G]|$ distinct switching equivalence classes of $G$-gain functions. A natural question is: which of these are in the same switching isomorphism class? It is easy to see that, in general, the automorphism group of a graph $\Gamma$ acts on the set of the switching equivalence classes: the orbits of this action are exactly the switching isomorphism classes of gain graphs whose underlying graph is isomorphic to $\Gamma$. The group of the automorphisms of $C_n$ is well known; it is isomorphic to the \emph{dihedral group $D_{2n}$} and its non-trivial elements can be partitioned   into \emph{reflections} and \emph{rotations}.
A rotation $\phi$ acts trivially on the switching equivalence classes, not affecting the conjugacy class of the gain of the closed walk $W_0$. More precisely:
$$
[\psi\circ \phi(W_0)]=[\psi(W_0)],\qquad \mbox{ for any $G$-gain function $\psi$ on $C_n$};
$$
thus  the switching equivalence classes are fixed by $\phi$.
On the contrary, a reflection $\sigma$ changes the orientation of the walk $W_0$, that is
$$
[\psi\circ \sigma(W_0)]=[\psi(W_0)^{-1}],\qquad \mbox{ for any $G$-gain function $\psi$ on $C_n$},
$$
and then  two gain functions  in the same switching isomorphism class can also have inverse gains on $W_0$. In particular, the switching isomorphisms classes of $C_n$ are in bijection with the conjugacy classes of $G$  up to group inversion. More explicitly,
 \begin{equation*}
\begin{split}
(C_n,\psi_1) \mbox{ and }  (C_n,\psi_2)\quad  \iff \qquad &[\psi_1(W_0)]=[\psi_2(W_0)] \\
\mbox{ are switching isomorphic }\qquad \quad\mbox{ or } &[\psi_1(W_0)]=[\psi_2(W_0)^{-1}].
\end{split}
\end{equation*}
Notice that if $G$ is such that $[g]=[g^{-1}]$ for every $g\in G$, then the switching isomorphism classes and the switching equivalence classes of $C_n$ coincide. A group with the property that every element is conjugate to its inverse is said to be \emph{ambivalent} \cite{ambi}, and the symmetric group $S_n$ is an example of such a group.
\subsubsection{$G$-cospectrality classes.} By virtue of Theorem \ref{teo:benposto}, two switching isomorphic $G$-gain graphs $(C_n,\psi_1)$ and  $(C_n,\psi_2)$ are $G$-cospectral.
\\We are going to prove that actually  $(C_n,\psi_1)$ and  $(C_n,\psi_2)$ are $G$-cospectral if and only if they are switching isomorphic. Suppose, by the contradiction,  that $(C_n,\psi_1)$ and  $(C_n,\psi_2)$ are $G$-cospectral but   they are not switching isomorphic. Let us set
 $a:=\psi_1(W_0)$ and  $b:=\psi_2(W_0)$ with $a,b\in G$. Since $(C_n,\psi_1)$ and  $(C_n,\psi_2)$ are non-switching isomorphic,
 it must be $[a]\neq[b]$ and $[a]\neq[b^{-1}]$. In particular, $a$ and $b$ cannot be both trivial; without loss of generality, we assume $a\neq 1_G$.

Since $(C_n,\psi_1)$ and  $(C_n,\psi_2)$ are $G$-cospectral, by  Proposition \ref{prop:cicli}, taking $h=n$ and $g=a$, one has:
\begin{equation}\label{eq:contra}
0<\left|\{W\in \mathcal C^n(C_n): \psi_1(W)\in [a]\} \right|
= \left|\{W\in \mathcal C^n(C_n): \psi_2(W)\in [a]\} \right|.
\end{equation}
But  in $\mathcal C^n(C_n)$ we have $2n$ closed walks visiting all vertices (as many as the pairs of centers and  orientations), and, when $n$ is even, we have also acyclic walks which are automatically balanced, since each edge is crossed the same number of times in the two opposite directions. Then
$$
\psi_2(W)\in[b]\cup[b^{-1}]\cup\{1_G\}
$$
for every $W\in \mathcal C^n(C_n)$. But $a\notin [b]\cup[b^{-1}]\cup\{1_G\}$, that is in contradiction with Eq.\ \eqref{eq:contra}.

\begin{corollary}\label{coro:det}
A $G$-gain graph whose underlying graph is a cycle is determined by its $G$-spectrum.
\end{corollary}
\begin{proof}
Let $(C_n,\psi)$ and $(\Gamma,\psi')$ be $G$-cospectral. In particular, they are $\pi_0$-cospectral, that is, the underlying graphs $C_n$ and $\Gamma$ are cospectral. Since $C_n$,  as (ungained) graph, is determined by its spectrum (e.g., \cite[Proposition 5]{deter}), also $\Gamma$ is a cycle. But we proved above that two $G$-cospectral gain graphs whose underlying  graph is a cycle, must be switching isomorphic.
\end{proof}

\subsubsection{$\lambda_G$-cospectrality classes.}
Here, we focus our attention on cospectrality of cyclic gain graphs with respect to the regular representation $\lambda_G$. We have already noticed that, as a consequence of \Cref{teo:pico}, two gain graphs $(\Gamma_1,\psi_1)$ and $(\Gamma_2,\psi_2)$ are $\lambda_G$-cospectral if and only if for every $h\in \mathbb N$, they have the same number of balanced closed walks of length $h$. When $\Gamma_1$ and $\Gamma_2$ are both isomorphic to $C_n$,  this simply means that $\psi_1(W_0)$ and $\psi_2(W_0)$ have the same \emph{order} in $G$, as the next proposition shows.

\begin{proposition}\label{prop:ordine}
Two gain graphs $(C_n,\psi_1)$ and $(C_n,\psi_2)$ are $\lambda_G$-cospectral if and only if $o(\psi_1(W_0))=o(\psi_2(W_0))$, where $W_0$ is the walk $v_1,v_2,\ldots,v_n,v_1$.
\end{proposition}
\begin{proof}
Suppose that $o(\psi_1(W_0))=o(\psi_2(W_0))$. The gain $\psi_i$ of a closed walk $W$ is conjugated with a (possibly negative, depending on the orientation) power of $\psi_i(W_0)$.
By assumption, this power of $\psi_i(W_0)$ is equal to $1_G$ for $i=1$ if and only if it is equal to $1_G$ for $i=2$. As a consequence $(C_n,\psi_1)$ and $(C_n,\psi_2)$ have the same number of balanced closed walks of length $h$ for every $h\in\mathbb N$,
and so they are $\lambda_G$-cospectral.\\
\indent In order to prove the converse implication, suppose now  that $k:=o(\psi_1(W_0))<o(\psi_2(W_0))$.
We claim that
\begin{equation}\label{eq:nuova}
 |\{W\in \mathcal C^{nk}(C_n): \psi_1(W)=1_G \}|\geq  |\{W\in \mathcal C^{nk}(C_n): \psi_2(W)=1_G \}|+2n.
 \end{equation}
More precisely $2n$ is the number of  closed walks which make exactly $k$ turns of the cycle (they are as many as the pairs of central vertices and orientations of the cycle):
these walks are balanced with respect to $\psi_1$ (their gains are conjugated with  $\psi_1(W_0)^k=1_G$ or its inverse)  while they are unbalanced with respect to $\psi_2$ (since $\psi_2(W_0)^k\neq1_G$).
Moreover, any closed walk of length $nk$ balanced for $\psi_2$ is automatically balanced for $\psi_1$.
\\From Eq.~\eqref{eq:nuova} it follows that $(C_n,\psi_1)$ and $(C_n,\psi_2)$ cannot be $\lambda_G$-cospectral.
\end{proof}

This characterization highlights that, even if $\lambda_G$ is a faithful representation containing  every irreducible representation of $G$, the information on a gain graph  given by its $\lambda_G$-spectrum  is quite poor.
 Certainly, the $\lambda_G$-cospectrality is weaker than $G$-cospectrality, as already announced in Remark \ref{rem:nonleft}.

\begin{example}\label{exa:2}
Let $\mathbb T_5=\{1_{\mathbb{T}_5},\xi,\xi^2,\xi^3,\xi^4\}$ be the cyclic  group of order $5$.
 Since $\mathbb T_5$ is Abelian,
the switching equivalence classes of  $\mathbb T_5$-gain functions on $C_n$ are $5$. Moreover, since $\xi^4 =\xi^{-1}$, and $\xi^3 = (\xi^2)^{-1}$, there are
$3$  switching isomorphism classes of $\mathbb T_5$-gain graphs whose underlying graph is $C_n$, and then also the number of  $\mathbb T_5$-cospectrality classes is $3$. Finally, every non-trivial element of $\mathbb T_5$ has order $5$: by Proposition \ref{prop:ordine}  there are only $2$ possible $\lambda_{\mathbb T_5}$-spectra for a $\mathbb T_5$-gain graph on $C_n$, one for the balanced case and the other for the unbalanced case.
\\More explicitly,  when a $\mathbb T_5$-gain graph $(C_n,\psi)$ is balanced, its cover graph is isomorphic to the disjoint union of $5$ copies of $C_n$ and the  $\lambda_{\mathbb T_5}$-spectrum  of $(C_n,\psi)$
is given by
 $$
 \left\{ \left(2\cos\frac{2\pi j}{n} \right)^{(5)},\; j=0,\ldots, n-1 \right\}.$$
 On the other hand, when  $(C_n,\psi)$ is  unbalanced one can check that its cover graph is isomorphic to the cyclic graph $C_{5n}$ and the $\lambda_{\mathbb T_5}$-spectrum  of $(C_n,\psi)$
 is given by
 $$\left\{ 2\cos \frac{2\pi j}{5n},\; j=0,\ldots, 5n-1 \right\}.$$

 In particular, there exists a pair $(C_n,\psi_1)$ and $(C_n,\psi_2)$ of $\lambda_{\mathbb T_5}$-cospectral graphs  that are not $\mathbb T_5$-cospectral.
And then, by Theorem \ref{teo:cosp}, there exists an irreducible representation $\pi$ such that
 $(C_n,\psi_1)$ and $(C_n,\psi_2)$ are not $\pi$-cospectral. This provides an example of gain graphs that are cospectral with respect to a sum of representations but that are not cospectral with respect to some addends, explicitly showing that the converse
 of  Proposition \ref{prop:somma} is not true. Notice that $(C_n,\psi_1)$ and $(C_n,\psi_2)$ are always $\pi_0$-cospectral; since any other irreducible representation of $\mathbb T_5$ is faithful, it follows that $(C_n,\psi_1)$ and $(C_n,\psi_2)$ are not $\pi$-cospectral for some faithful representation $\pi$.
\end{example}

\subsection{$\mathbb T_m$-gain graphs}
As we have shown in the previous example,  cospectrality with respect to a faithful representation does not imply cospectrality with respect to any other faithful representation. What happens if one restricts to faithful irreducible representations?
In this subsection we investigate  this question for the cyclic group
$$
\mathbb T_m=\{1_{\mathbb T_m},\xi,\xi^2,\ldots, \xi^{m-1}\}
$$
or order $m$, which is  isomorphic to the group of $m$-th roots of unity.
 An element $f=\sum_{i=0}^{m-1}  f_{\xi^i} \,\xi^i \in \mathbb C \mathbb T_m$ can be identified with a polynomial of degree at most $m-1$ with coefficients in $\mathbb C$:
$$ f(z):=\sum_{i=0}^{m-1} f_{\xi^i
} z^{i}.$$
Two $\mathbb T_m$-gain graphs $(\Gamma_1,\psi_1)$ and $(\Gamma_2,\psi_2)$ with $A:=A_{(\Gamma_1,\psi_1)}$ and $B:=A_{(\Gamma_2,\psi_2)}$, are
$\mathbb T_m$-cospectral if and only if, for every $h\in \mathbb N$, the polynomials $Tr(A^h)$ and $Tr(B^h)$ coincide (see Definition \ref{def:cosp} and remember that $\mathbb T_m$ is Abelian). Thus $(\Gamma_1,\psi_1)$ and $(\Gamma_2,\psi_2)$ are $\mathbb T_m$-cospectral
if and only if, for every $h\in \mathbb N$, the polynomial
\begin{equation}\label{eq:pp}
P_h(z):= Tr(A^h)(z)-Tr(B^h)(z)
\end{equation}
is the zero polynomial.

The unitary irreducible representations of $\mathbb T_m$ have degree $1$ and each of them maps the generator $\xi$ to an $m$-th root of unit in $\mathbb C$. More precisely,
a complete system of  unitary irreducible representations of $\mathbb T_m$ is given by  $\pi_0,\pi_1,\ldots, \pi_{m-1}$, where for each $j$ one has:
\begin{equation}\label{eq:defi}
\pi_j(\xi)=e^{ \frac{2\pi ij}{m}}\in\mathbb C.
\end{equation}
Actually $\mathbb T_m$ can be seen already embedded in $\mathbb T$ (with $\xi=e^{ \frac{2\pi i}{m}}$), and in this case the representation $\pi_1$ coincides with $\pi_{id}$.
Notice that, for each $j=0,\ldots,m-1$, the linear extension of $\pi_j$ to $\mathbb C \mathbb T_m$ is such that:
\begin{equation}\label{eq:valu}
\pi_j\left(\sum_{i=0}^{m-1}  f_{\xi^i} \,\xi^i\right)=\sum_{i=0}^{m-1} f_{\xi^i}\pi_j(\xi^i)=\sum_{i=0}^{m-1} f_{\xi^i}\pi_j(\xi)^i=f(\pi_j(\xi)), \qquad \forall f\in\mathbb C \mathbb T_m.
\end{equation}

We are going to analyze   $\pi_j$-cospectrality of  $(\Gamma_1,\psi_1)$ and $(\Gamma_2,\psi_2)$ when $\pi_j$ is faithful.
A representation $\pi_j$ is faithful if and only if $\pi_j(\xi)$ is an $m$-th primitive root, that is, if and only if $j$ and $m$ are coprime.
The Eq.~\eqref{eq:inteo} of Theorem \ref{teo:pico} for the representation $\pi_j$  is equivalent, by linearity, to
\begin{equation}\label{eq:inpi}
\pi_j\left(Tr\left(A^h\right)\right)=\pi_j\left(Tr\left(B^h\right)\right),\qquad \forall h\in\mathbb N.
\end{equation}
Combining \Cref{eq:valu} with \Cref{eq:inpi}, recalling the definition of the polynomial $P_h$ from \Cref{eq:pp}, one has that $(\Gamma_1,\psi_1)$ and $(\Gamma_2,\psi_2)$ are $\pi_j$-cospectral if and only if
$$P_h( \pi_j(\xi))=0, \quad \forall h\in\mathbb N.$$

The \emph{cyclotomic polynomial $\Phi_m$} is irreducible on $\mathbb Z$ and its roots are the $m$-th primitive roots of unit.
Notice that for each $h$, the polynomial $P_h$ has integer coefficients, this implies that if an $m$-th primitive root is a root of $P_h$, then the cyclotomic polynomial $\Phi_m$ divides $P_h$ and then  any other  $m$-th primitive root is a root of $P_h$.
The next corollary directly follows.
\begin{corollary}\label{coro:ultimo}
Let  $(\Gamma_1,\psi_1)$ and $(\Gamma_2,\psi_2)$ be two $\mathbb T_m$-gain graphs. Let $\pi$ and $\pi'$ be two unitary, faithful, irreducible representations of $\mathbb T_m$. Then
$$(\Gamma_1,\psi_1)\mbox{ and }(\Gamma_2,\psi_2) \mbox{ are $\pi$-cospectral } \iff (\Gamma_1,\psi_1)\mbox{ and }(\Gamma_2,\psi_2) \mbox{ are $\pi'$-cospectral}.$$
In particular, when $m$ is prime, the gain graphs $(\Gamma_1,\psi_1)$ and $(\Gamma_2,\psi_2)$ are $\mathbb T_m$-cospectral if and only if they are $\pi_{id}$-cospectral and their underlying graphs $\Gamma_1$ and $\Gamma_2$ are cospectral.
\end{corollary}
\begin{proof}
The representations $\pi$ and $\pi'$ are unitary, faithful, irreducible by hypothesis. Therefore we can assume $\pi=\pi_j$ and $\pi'=\pi_{j'}$ (see \Cref{eq:defi}), with
$j,j'\in\{1,\ldots,m-1\}$ both coprime with $m$. In particular $\pi_j(\xi)$ and $\pi_{j'}(\xi)$ are both $m$-th primitive roots, and then, for every $h\in\mathbb N$,
$$P_h(\pi_j(\xi))=0\iff P_h(\pi_{j'}(\xi))=0,$$
and then $(\Gamma_1,\psi_1)\mbox{ and }(\Gamma_2,\psi_2)$ are $\pi_{j}$-cospectral if and only if they are $\pi_{j'}$-cospectral.
\\ \indent The second part of the statement follows from the fact that, when $m$ is prime,  all non-trivial irreducible representations are faithful, and from the fact that the $\pi_0$-cospectrality is equivalent to the cospectrality of the underlying graphs.
\end{proof}

\begin{example}\label{exa:nuovofine}
Let $\mathbb T_5=\{1_{\mathbb T_5},\xi,\xi^2,\xi^3,\xi^4\}$ be the cyclic  group of order $5$.
Consider the $\mathbb T_5$-gain graphs $(\Gamma_1,\psi_1)$ and $(\Gamma_2,\psi_2)$ depicted in Fig. \ref{fig:uu}, where the gain of each unlabelled edge    is $1_{\mathbb T_5}$ and where if an oriented edge from a vertex $u$ to a vertex $v$ in $(\Gamma_i,\psi_i)$ has label $\xi$ then $\psi_i(u,v)=\xi$ and $\psi_i(v,u)=\xi^{-1}$.
\\The underlying graphs $\Gamma_1$ and $\Gamma_2$ are cospectral, both with characteristic polynomial equal to:
$$(x-1)(x^3-x^2-5x+1)(x+1)^2.$$
Moreover, $(\Gamma_1,\psi_1)$ and $(\Gamma_2,\psi_2)$ are $\pi_{id}$-cospectral, since both the matrices
$$\begin{pmatrix}
0&1&0&0&0&0\\
1&0&e^{\frac{2\pi i}{5}}&0&1&0\\
0&e^{-\frac{2\pi i}{5}}&0&1&1&0\\
0&0&1&0&e^{-\frac{2\pi i}{5}}&1\\
0&1&1&e^{\frac{2\pi i}{5}}&0&0\\
0&0&0&1&0&0\\
\end{pmatrix} \mbox{ and }
\begin{pmatrix}
0&1&1&1&1&1\\
1&0&0&0&0&0\\
1&0&0&e^{\frac{2\pi i}{5}}&0&0\\
1&0&e^{-\frac{2\pi i}{5}}&0&0&0\\
1&0&0&0&0&e^{\frac{2\pi i}{5}}\\
1&0&0&0&e^{-\frac{2\pi i}{5}}&0\\
\end{pmatrix}
$$
have characteristic polynomial:
$$\left(x^4-6x^2-4x\cos\frac{2\pi}{5}+1\right)(x^2-1).$$
By \Cref{coro:ultimo}, this is enough to state that $(\Gamma_1,\psi_1)$ and $(\Gamma_2,\psi_2)$ are $\mathbb T_5$-cospectral.
Notice that this implies also that $(\Gamma_1,\psi_1)$ and $(\Gamma_2,\psi_2)$ are $\mathbb T$-cospectral as $\mathbb T$-gain graphs.
\begin{center}
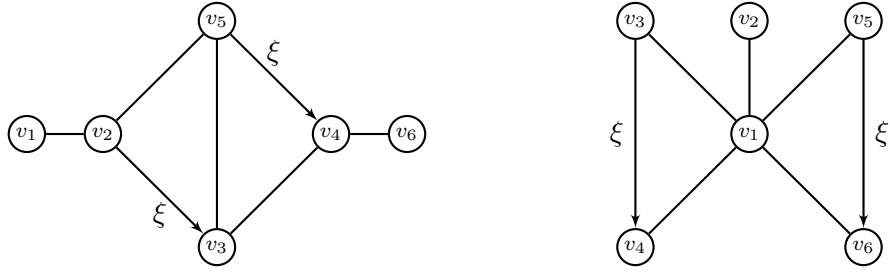
\begin{figure}
\begin{tikzpicture}


\node[vertex]  (1) at (0.5,0){\tiny $v_1$};

\node[vertex]  (2) at (1.5,0){\tiny $v_2$};

\node[vertex]  (3) at (3,-1.5){\tiny $v_3$};

\node[vertex]  (4) at (4.5,0){\tiny $v_4$};

\node[vertex]  (5) at (3,1.5){\tiny $v_5$};

\node[vertex]  (6) at (5.5,0){\tiny $v_6$};


\draw[edge] (1) -- (2);

\draw[dedge] (2) -- (3)node[midway, below] {\small $\xi$};

\draw[edge] (3) -- (4);

\draw[edge] (3) -- (5);

\draw[dedge] (5) -- (4)node[midway, above] {\small $\xi$};

\draw[edge] (5) -- (2);

\draw[edge] (4) -- (6);

\end{tikzpicture}
\hspace{2cm}
\begin{tikzpicture}


\node[vertex]  (1) at (0,0){\tiny $v_1$};

\node[vertex]  (2) at (0,1.5){\tiny $v_2$};

\node[vertex]  (3) at (-1.5,1.5){\tiny $v_3$};

\node[vertex]  (4) at (-1.5,-1.5){\tiny $v_4$};

\node[vertex]  (5) at (1.5,1.5){\tiny $v_5$};

\node[vertex]  (6) at (1.5,-1.5){\tiny $v_6$};


\draw[edge] (1) -- (2);
\draw[edge] (1) -- (3);
\draw[edge] (1) -- (4);
\draw[edge] (1) -- (5);
\draw[edge] (1) -- (6);
\draw[dedge] (3) -- (4)node[midway, left] {\small $\xi$};
\draw[dedge] (5) -- (6)node[midway, right] {\small $\xi$};

\end{tikzpicture}\caption{The $\mathbb T_5$-gain graphs $(\Gamma_1,\psi_1)$ and $(\Gamma_2,\psi_2)$ of Example \ref{exa:nuovofine}. }\label{fig:uu}
\end{figure}
\end{center}
\end{example}
In \Cref{coro:ultimo} we proved that two  $\mathbb T_m$-gain graphs are cospectral with respect to either all irreducible faithful representations or  none of them.
It is natural to ask whether this behavior is not peculiar of finite subgroups of $\mathbb T$ and if it is more generally valid, maybe even for non-Abelian groups. A positive answer  would suggest to consider a milder definition of cospectrality for $G$-gain graphs, less strong than $G$-cospectrality, but still independent from the choice of the representation (provided it is faithful and irreducible), which would be also more coherent with the ``classical'' concept of cospectrality in $\mathbb T$-gain graphs and signed graphs, where only the $\pi_{id}$-spectrum is considered. The answer is no: there are pairs of gain graphs, cospectral with respect to a  faithful irreducible representation and non-cospectral
with respect to another  faithful irreducible representation.
In order to show that, we are going to investigate a group with a richer  representation theory, such as the symmetric  group $S_4$ is.

\subsection{An example over the symmetric group $S_4$}
The symmetric group $S_4$ admits two faithful irreducible representations: the standard representation and its tensor product with the alternating representation \cite[Section~2.3]{fulton}, both of degree $3$.
We explicitly construct  unitary representatives $\pi_{St}$ and $\pi_{St\otimes A}$ for both representations in \Cref{table}.

\begin{center}
\begin{tabular}{c|c|c|}
& $(12)$& $(1234)$\\
\hline
$\pi_{St}$ &
$\begin{pmatrix}
-1&0&0\\
0&1&0\\
0&0&1\\
\end{pmatrix}$  &
$\begin{pmatrix}
-\frac{1}{2}&\frac{\sqrt{3}}{2}&0\\
-\frac{\sqrt{3}}{6}&-\frac{1}{6}& \frac{2\sqrt{2}}{3}&\\
-\frac{\sqrt{6}}{3}&-\frac{\sqrt{2}}{3}&-\frac{1}{3}\\
\end{pmatrix}$ \\
\hline
$\pi_{St\otimes A}$ &
$\begin{pmatrix}
1&0&0\\
0&-1&0\\
0&0&-1\\
\end{pmatrix}$  &
$\begin{pmatrix}
\frac{1}{2}&-\frac{\sqrt{3}}{2}&0\\
\frac{\sqrt{3}}{6}&\frac{1}{6}& -\frac{2\sqrt{2}}{3}&\\
\frac{\sqrt{6}}{3}&\frac{\sqrt{2}}{3}&\frac{1}{3}\\
\end{pmatrix}$ \\
\hline
\end{tabular}
\captionof{table}{The images of $(12)$ and $(1234)$ via the representations $\pi_{St}$ and $\pi_{St\otimes A}$.}\label{table}
\end{center}

The permutations $(12)$ and $(1234)$  generate $S_4$ and one can check  that  $\pi_{St}$ and $\pi_{St\otimes A}$ extend to  homomorphisms from $S_4$ to $U_3(\mathbb{C})$.
In other words, by using \Cref{table}, one can compute the values of $\pi_{St}$ and $\pi_{St\otimes A}$ on any permutation in $S_4$
(with the convention of multiplying from the left to the right).
For example, from the relation $(34)=(1234)^{-2}(12)(1234)^{2}$, we obtain:
\begin{equation*}\label{eq:34}
\begin{split}
 \pi_{St}((34))=  \pi_{St}((1234))^{-2} \pi_{St}((12))  \pi_{St}((1234))^2&=\begin{pmatrix}
1&0&0\\
0&\frac{1}{3}& \frac{2\sqrt{2}}{3}&\\
0& \frac{2\sqrt{2}}{3}&-\frac{1}{3}\\
\end{pmatrix};\\
\pi_{St\otimes A}((34))=  \pi_{St\otimes A}((1234))^{-2} \pi_{St\otimes A}((12))  \pi_{St\otimes A}((1234))^2&=\begin{pmatrix}
-1&0&0\\
0&-\frac{1}{3}& -\frac{2\sqrt{2}}{3}&\\
0& -\frac{2\sqrt{2}}{3}&\frac{1}{3}\\
\end{pmatrix}.
\end{split}
\end{equation*}
Also:
\begin{equation*}\label{eq:1234}
\begin{split}
 \pi_{St}((12)(34))=  \pi_{St}((12)) \pi_{St}((34)) &=
 \begin{pmatrix}
-1&0&0\\
0&\frac{1}{3}& \frac{2\sqrt{2}}{3}&\\
0& \frac{2\sqrt{2}}{3}&-\frac{1}{3}\\
\end{pmatrix};\\
 \pi_{St\otimes A}((12)(34))=  \pi_{St\otimes A}((12)) \pi_{St\otimes A}((34)) &=\begin{pmatrix}
-1&0&0\\
0&\frac{1}{3}& \frac{2\sqrt{2}}{3}&\\
0& \frac{2\sqrt{2}}{3}&-\frac{1}{3}\\
\end{pmatrix}.
\end{split}
\end{equation*}

\begin{example}\label{exa:s4}
Let $(\Gamma,\psi)$ be the $S_4$-gain graph depicted in Fig.~\ref{fig:2}, whose group algebra valued adjacency matrix is
$$A_{(\Gamma,\psi)}=\begin{pmatrix}
0&(12)(34)&0&1_{S_4}&0&0&0&0&0\\
(12)(34)&0&1_{S_4}&0&1_{S_4}&0&0&0&0\\
0&1_{S_4}&0&(12)(34)&1_{S_4}&0&0&0&0\\
1_{S_4}&0&(12)(34)&0&1_{S_4}&0&0&0&0\\
0&1_{S_4}&1_{S_4}&1_{S_4}&0&1_{S_4}&0&0&0\\
0&0&0&0& 1_{S_4}&0&(12)&0&(34)\\
0&0&0&0& 0&(12)&0&(34)&0\\
0&0&0&0& 0&0&(34)&0& (12)\\
0&0&0&0& 0&(34)&0&(12)&0\\
\end{pmatrix}.
$$
Let $(\Gamma,\psi')$ be the $S_4$-gain graph depicted in Fig.~\ref{fig:3}, whose group algebra valued adjacency matrix is
$$A_{(\Gamma,\psi')}=\begin{pmatrix}
0&(34)&0&(12)&0&0&0&0&0\\
(34)&0&(12)&0&1_{S_4}&0&0&0&0\\
0&(12)&0&(34)&1_{S_4}&0&0&0&0\\
(12)&0&(34)&0&1_{S_4}&0&0&0&0\\
0&1_{S_4}&1_{S_4}&1_{S_4}&0&1_{S_4}&0&0&0\\
0&0&0&0& 1_{S_4}&0&1_{S_4}&0&(12)(34)\\
0&0&0&0& 0&1_{S_4}&0&(12)(34)&0\\
0&0&0&0& 0&0&(12)(34)&0& 1_{S_4}\\
0&0&0&0& 0&(12)(34)&0&1_{S_4}&0\\
\end{pmatrix}.
$$
Notice that each unlabelled edge in Fig. \ref{fig:2} and Fig. \ref{fig:3} corresponds to an edge whose gain is $1_{S_4}$, and in both graphs each gain is an involution in $S_4$ and then it is not necessary to specify the edge direction associated with the label.

By a direct computation, one can check that the $27\times 27$ matrices $\pi_{St}(A_{(\Gamma,\psi)})$ and $\pi_{St}(A_{(\Gamma,\psi')})$ are cospectral but $\pi_{St\otimes A}(A_{(\Gamma,\psi)})$ and $\pi_{St\otimes A}(A_{(\Gamma,\psi')})$ are not.
More precisely, $\pi_{St}(A_{(\Gamma,\psi)})$ and $\pi_{St}(A_{(\Gamma,\psi')})$  and  $\pi_{St\otimes A}(A_{(\Gamma,\psi)})$
have the same characteristic polynomial, that is:
\begin{equation*}
\begin{split}
 (x^9-12x^7+44x^5-48x^3)\cdot (&x^{18}-24x^{16}-4x^{15}
+224x^{14}+64x^{13}-1024x^{12}-368x^{11}\\&+2352x^{10}+896x^9-2432x^8-768x^7+768x^6).
\end{split}
\end{equation*}
On the other hand, the characteristic polynomial  of $\pi_{St\otimes A}(A_{(\Gamma,\psi')})$  is:
\begin{equation*}
\begin{split}
 (x^9-12x^7+44x^5-48x^3)\cdot (&x^{18}-24x^{16}+4x^{15}
+224x^{14}-64x^{13}-1024x^{12}+368x^{11}\\&+2352x^{10}-896x^9-2432x^8+768x^7+768x^6).
\end{split}
\end{equation*}
Therefore $(\Gamma,\psi)$ and $(\Gamma,\psi')$ are $\pi_{St}$-cospectral but they are not  $\pi_{St\otimes A}$-cospectral.
Notice that the idea behind the construction is that the element $1_{S_4}+(12)(34)-(12)-(34)\in\mathbb C S_4$ is in the kernel of the extension of $\pi_{St}$ to $\mathbb C S_4$
(it is also in the kernel of the extension of the natural permutation representation, that is equivalent to $\pi_{0}\oplus \pi_{St}$).
But the same element is not in the kernel of the extension of
$\pi_{St\otimes A}$. More precisely:
\begin{equation}\label{eq:finale}
\begin{split}
\pi_{St}(1_{S_4})+\pi_{St}((12)(34))&=\pi_{St}((12))+\pi_{St}((34))\\
\pi_{St\otimes A}(1_{S_4})+\pi_{St\otimes A}((12)(34))&=-\pi_{St\otimes A}((12))-\pi_{St\otimes A}((34)).
\end{split}
\end{equation}
This way, looking at the gains in $(\Gamma,\psi)$ and $(\Gamma,\psi')$ of the two triangles, one has
$$\sum_{W \in\mathcal C^3} \chi_{\pi_{St}}(\psi(W))= \sum_{W \in\mathcal C^3} \chi_{\pi_{St}}(\psi'(W))$$
but
$$\sum_{W \in\mathcal C^3} \chi_{\pi_{St\otimes A}}(\psi(W))\neq \sum_{W \in\mathcal C^3} \chi_{\pi_{St\otimes A}}(\psi'(W)),$$
and then, for the representation $\pi_{St\otimes A}$,  the condition of Eq.~\eqref{eq:inteo} does not hold for $h=3$; it follows from Theorem \ref{teo:pico} that $(\Gamma,\psi)$ and $(\Gamma,\psi')$  are not  $\pi_{St\otimes A}$-cospectral.
\end{example}

\begin{center}
\begin{figure}
\begin{tikzpicture}
\node[vertex]  (1) at (0,0){\tiny $v_1$};
\node[vertex]  (2) at (2,2) {\tiny $v_2$};
\node[vertex]  (3) at (4,0) {\tiny $v_3$};
\node[vertex]  (4) at (2,-2) {\tiny $v_4$};
\node[vertex]  (5) at (6,0) {\tiny $v_5$};
\node[vertex]  (6) at (8,0){\tiny $v_6$};
\node[vertex]  (7) at (10,2) {\tiny $v_7$};
\node[vertex]  (8) at (12,0) {\tiny $v_8$};
\node[vertex]  (9) at (10,-2) {\tiny $v_9$};

\draw[edge] (1) -- (2) node[midway, above,sloped] {\small $(12)(34)$};
\draw[edge] (3) -- (4) node[midway, above, sloped] {\small $(12)(34)$};
\draw[edge] (3) -- (2);
\draw[edge] (1) -- (4);
\draw[edge] (2) -- (5);
\draw[edge] (3) -- (5);
\draw[edge] (4) -- (5);
\draw[edge] (6) -- (5);
\draw[edge] (6) -- (7) node[midway, above,sloped] {\small $(12)$};
\draw[edge] (8) -- (9) node[midway, above, sloped] {\small $(12)$};
\draw[edge] (8) -- (7) node[midway, above,sloped] {\small $(34)$};
\draw[edge] (6) -- (9) node[midway, above, sloped] {\small $(34)$};

\end{tikzpicture}\caption{The $S_4$-gain graph $(\Gamma,\psi)$ of Example \ref{exa:s4}. }\label{fig:2}
\vspace{1cm}
\begin{tikzpicture}

\node[vertex]  (1) at (0,0){\tiny $v_1$};
\node[vertex]  (2) at (2,2) {\tiny $v_2$};
\node[vertex]  (3) at (4,0) {\tiny $v_3$};
\node[vertex]  (4) at (2,-2) {\tiny $v_4$};
\node[vertex]  (5) at (6,0) {\tiny $v_5$};
\node[vertex]  (6) at (8,0){\tiny $v_6$};
\node[vertex]  (7) at (10,2) {\tiny $v_7$};
\node[vertex]  (8) at (12,0) {\tiny $v_8$};
\node[vertex]  (9) at (10,-2) {\tiny $v_9$};

\draw[edge] (1) -- (4) node[midway, below,sloped] {\small $(12)$};
\draw[edge] (3) -- (4) node[midway, above, sloped] {\small $(34)$};
\draw[edge] (2) -- (3)node[midway, below,sloped] {\small $(12)$};
\draw[edge] (1) -- (2)node[midway, above, sloped] {\small $(34)$};
\draw[edge] (2) -- (5);
\draw[edge] (3) -- (5);
\draw[edge] (4) -- (5);
\draw[edge] (6) -- (5);
\draw[edge] (6) -- (7);
\draw[edge] (8) -- (9);
\draw[edge] (8) -- (7) node[midway, above,sloped] {\small $(12)(34)$};
\draw[edge] (6) -- (9)node[midway, above,sloped] {\small $(12)(34)$};

\end{tikzpicture}\caption{The $S_4$-gain graph $(\Gamma,\psi')$ of Example \ref{exa:s4}. }\label{fig:3}
\end{figure}
\end{center}

\end{document}